\documentclass[11pt]{amsart}
\usepackage[utf8]{inputenc}
\usepackage{bm}
\usepackage{fontenc}
\usepackage{amsfonts}
\usepackage{amssymb}
\usepackage{amsmath}
\usepackage{amsthm}\usepackage{mathtools}
\usepackage{enumerate}
\usepackage{hyperref}\usepackage{enumitem}
\usepackage{mathrsfs}
\usepackage{tikz}\usetikzlibrary{calc}
\usepackage{marginnote}
\usepackage{xcolor,enumitem}
\usepackage{soul}\usepackage{tikz}
\usepackage[all,cmtip]{xy}

\newcommand{\rs}{\mathord{\upharpoonright}}

\newcommand{\e}{\varepsilon}

\newcommand{\Q}{\mathbb{Q}}
\newcommand{\R}{\mathbb{R}}
\newcommand{\Z}{\mathbb{Z}}
\newcommand{\D}{\mathbb{D}}
\newcommand{\N}{\mathbb{N}}
\newcommand{\C}{\mathbb{C}}

\newcommand{\cY}{\mathcal{Y}}

\newcommand{\cE}{\mathcal{E}}

\newcommand{\eps}{\varepsilon}

\newcommand{\SOTh}{\mathrm{SOT}\text{-}}
\newcommand{\WOTh}{\mathrm{WOT}\text{-}}

\newcommand{\NN}{\mathbb{N}}

\newcommand{\cstu}{\mathrm{C}^*_u}

\newtheorem*{rigprob*}{Rigidity Problem for uniform Roe Algebras}
\newtheorem*{rigprobcorona*}{Rigidity Problem for uniform Roe Coronas}

\newcommand{\cst}{\mathrm{C}^*}
\newcommand{\cstar}{$\mathrm{C}^*$}

\newcommand{\cZ}{\mathcal{Z}}

\newcommand{\cP}{\mathcal{P}}

\newcommand{\cB}{\mathcal{B}}
\newcommand{\cK}{\mathcal{K}}

\newcommand{\cD}{\mathcal D} 
 
\newcommand{\cA}{\mathcal A} 

\newcommand{\Out}{\mathrm{Out}}

\newcommand{\Coa}{\mathrm{Coa}}
\newcommand{\BijCoa}{\mathrm{BijCoa}}
\newcommand{\Aut}{\mathrm{Aut}}
\newcommand{\Inn}{\mathrm{Inn}}
\newcommand{\BD}{\mathrm{BD}}

\newtheorem{theorem}{Theorem}[section]
\newtheorem*{theorem*}{Theorem}
\newtheorem{proposition}[theorem]{Proposition}

\newtheorem*{proposition*}{Proposition}
\newtheorem{lemma}[theorem]{Lemma}
\newtheorem*{lemma*}{Lemma}
\newtheorem{corollary}[theorem]{Corollary}
\newtheorem*{corollary*}{Corollar}

\newtheorem*{fact*}{Fact}
\theoremstyle{definition}
\newtheorem{definition}[theorem]{Definition}
\newtheorem*{definition*}{Definition}
\newtheorem{claim}[theorem]{Claim}
\newtheorem*{claim*}{Claim}

\newtheorem*{conjecture*}{Conjecture}

\newtheorem{theoremi}{Theorem}
\newtheorem{problemi}{Problem}
\newtheorem{assumption}[theorem]{Assumption}

\theoremstyle{remark}

\newtheorem*{example*}{Example}
\newtheorem{remark}[theorem]{Remark}
\newtheorem*{remark*}{Remark}

\newtheorem*{note*}{Note}
\newtheorem*{question*}{Question}

\newtheorem*{acknowledgements*}{Acknowledgements}
\newcommand{\norm}[1]{\left\lVert #1 \right\rVert}

\DeclareMathOperator{\Span}{span}

\DeclareMathOperator{\supp}{supp}

\DeclareMathOperator{\propg}{prop}

\DeclareMathOperator{\Proj}{Proj}

\DeclareMathOperator{\rank}{rank}

\DeclareMathOperator{\Ad}{Ad}

\DeclareMathOperator{\diam}{diam}

\DeclareMathOperator{\spann}{span}

\newcounter{my_enumerate_counter}
\newcommand{\pushcounter}{\setcounter{my_enumerate_counter}{\value{enumi}}}
\newcommand{\popcounter}{\setcounter{enumi}{\value{my_enumerate_counter}}}

\usepackage{enumitem}

\begin{document}

\title[A Gelfand-type duality for coarse metric spaces ]{A Gelfand-type duality for coarse metric spaces with property A}%
\author[B. M. Braga]{Bruno M. Braga}
\address[B. M. Braga]{University of Virginia, 141 Cabell Drive, Kerchof Hall, P.O. Box 400137, Charlottesville, USA}
\email{demendoncabraga@gmail.com}
\urladdr{https://sites.google.com/site/demendoncabraga}

\author[A. Vignati]{Alessandro Vignati}
\address[A. Vignati]{
Institut de Math\'ematiques de Jussieu - Paris Rive Gauche (IMJ-PRG)\\
Universit\'e Paris Diderot\\
B\^atiment Sophie Germain\\
8 Place Aur\'elie Nemours \\ 75013 Paris, France}
\email{ale.vignati@gmail.com}
\urladdr{http://www.automorph.net/avignati}

\subjclass[2010]{46L05, 20F65,	51F99}
\keywords{}
\thanks{ }
\date{\today}%
 
\begin{abstract} We prove the following two results for a given uniformly locally finite metric space with Yu's property A:
\begin{enumerate}
\item The group of outer automorphisms of its uniform Roe algebra is isomorphic to its group of bijective coarse equivalences modulo closeness.
\item The group of outer automorphisms of its Roe algebra is isomorphic to its group of coarse equivalences modulo closeness.
\end{enumerate}
The main difficulty lies in the latter. To prove that, we obtain several uniform approximability results for maps between Roe algebras and use them to obtain a theorem about the `uniqueness' of Cartan masas of Roe algebras. We finish the paper with several applications of the results above to concrete metric spaces. 
\end{abstract}
\maketitle

%\setcounter{tocdepth}{1}
%\tableofcontents

\section{Introduction}\label{SecIntro}

Given the class of metric spaces, consider the following three kinds of morphisms: (1) homeomorphisms --- maps preserving the topological structure ---, (2) coarse equivalences --- maps preserving the large-scale geometry --- and (3) bijective coarse equivalences. In short, coarse equivalences uniformly send close points to close points, far points to far points, and have
large image in their codomain. Although, the set of homeomorphisms of a metric space forms a group under composition, this is not the case for coarse equivalences. Indeed, coarse equivalences need to be neither injective nor surjective. However, the set of coarse equivalences on a metric space $(X,d)$ becomes a group after identifing coarse equivalences which are \emph{close} to each other (see Definition~\ref{def:close}). We denote by $\Coa(X)$ the group of all coarse equivalences of $X$ modulo the closeness relation and by $\BijCoa(X)$ the group of all bijective coarse equivalences of $X$ modulo closeness (we refer the reader to \S\ref{SecPrelim} for details).

In case $X$ is locally compact, homeomorphisms correspond, thanks to Gelfand's transform, to automorphisms of the \cstar-algebra of continuous functions on $X$ vanishing at infinity, $C_0(X)$. The goal of this paper is to give an operator algebraic characterisation of the groups $\Coa(X)$ and $\BijCoa(X)$, at least when dealing with metric spaces with certain regularity properties. The objects apt to this coarse Gelfand-type correspondence are Roe-type \cstar-algebras. These \cstar-algebras were introduced by Roe in \cite{Roe1993} for their connections to (higher) index theory and the associated applications to manifold topology and geometry (\cite{Roe1996}). The Roe algebra and its uniform version were studied precisely to detect \cstar-algebraically the large-scale geometry of metric spaces. Their study was boosted due to their intrinsic relation with the coarse Baum-Connes conjecture and consequently with the coarse Novikov conjecture (\cite{Yu2000}). Recently, Roe-type algebras and their $K$-theory have been used as a framework in mathematical physics to study the classification of topological phases and the topology of quantum systems (\cite{EwertMeyer2019,Kubota2017}). 

We now describe our main results. Given a metric space $(X,d)$ and a Hilbert space $H$, $\ell_2(X,H)$ denotes the Hilbert space of square summable $H$-valued functions on $X$. Operators in $\mathcal B(\ell_2(X,H))$ can be seen as $X\times X$-matrices whose entries are in $\mathcal B(H)$. Given an operator $a=(a_{xy})_{x,y\in X}$ in $\mathcal B(\ell_2(X,H))$, we define its propagation as the quantity 
\[
\propg(a)=\sup \{d(x,y)\mid a_{xy}\neq 0\}.
\]
 If $H$ is separable and infinite-dimensional, the \emph{\cstar-algebra of band-dominated operators of $(X,d)$}, denoted by $\BD(X)$, is the norm closure of the $^*$-algebra of finite propagation operators. If in addition we demand each entry $a_{xy}$ to be compact, we obtain the \emph{Roe algebra of $X$}, $\cst(X)$. Finally, if $H=\mathbb C$, the \emph{uniform Roe algebra of $X$}, $\cstu(X)$, is defined once again as the norm closure of the $^*$-algebra of finite propagation operators on $\ell_2(X)=\ell_2(X,\C)$ (see \S\ref{SecIntro} for more details on those definitions).

Let us first focus on the case of $\BijCoa(X)$. A bijective coarse equivalence of $X$ induces an automorphism of $\cstu(X)$ in a canonical way, and two bijective coarse equivalences are close if and only if the associated isomorphisms are unitarily equivalent in $\cstu(X)$ (see \S\ref{SubsectionCanonicalMap}). This gives a well defined canonical group monomorphism from $\BijCoa(X)$ into $\Out(\cstu(X))$, the latter being the group of outer automorphisms of $\cstu(X)$, i.e., 
\[
\Out(\cstu(X))=\Aut(\cstu(X))/\Inn(\cstu(X)).
\]
\begin{problemi}[Gelfand-type duality for bijective coarse equivalences] 
Let $X$ be a uniformly locally finite metric space. Is the canonical homomorphism
\[\BijCoa(X)\to \Out(\cstu(X))\]
a group isomorphism?\label{ProblemA}
\end{problemi}

The work of White and Willett on uniqueness of Cartan masas in uniform Roe algebras in presence of Yu's property A can be used to give a positive answer to Problem \ref{ProblemA} (again in the presence of property A). The following, proven as Theorem \ref{ThmIsoCoarseAutURA}, is a consequence of \cite[Theorem E]{WhiteWillett2017}: 
 
\begin{theoremi}\label{thm:uniform}
Let $(X,d)$ be a uniformly locally finite metric space with property A. The canonical homomorphism 
\[
\BijCoa(X)\to\Out(\cstu(X))
\]
is a group isomorphism.
\end{theoremi}

Theorem \ref{thm:uniform} gives an alternative way to compute the outer automorphism group of a uniform Roe algebra. As a simple application, it can be used to show that all automorphisms of $\cstu(\N)$ are inner and that $\Out(\cstu(\Z))\cong\Z_2$ (see Corollary \ref{CorOutNandZ}).

 Let us now focus on the case of coarse equivalences. Although coarse equivalences do not induce uniform Roe algebra isomorphisms (for instance, all finite metric spaces are coarsely equivalent, but, if $X$ and $Y$ are finite, $\cstu(X)$ and $\cstu(Y)$ are isomorphic if and only if $|X|=|Y|$), they do induce isomorphisms between Roe algebras. If $X$ is a uniformly locally finite metric space, assigning an element of $\Aut(\cst(X))$ to a coarse equivalence of $X$ is highly non-canonical (see \S\ref{SubsectionCanonicalMap}). Such an assignment becomes canonical when considered as a map from $\Coa(X)$ to $\Aut(\cst(X))$ modulo $\Inn(\BD(X))$ --- notice that there are a couple of hidden claims in here: (1) we are allowed to mod out our maps and (2) $\Inn(\BD(X))$ is a normal subgroup of $\Aut(\cst(X))$ (the latter follows since we prove that $\BD(X)$ is the multiplier algebra of $\cst(X)$, see Theorem~\ref{prop:mult}). We form the outer automorphism group of $\cst(X)$ by letting 
\[
\Out(\cst(X))=\Aut(\cst(X))/\Inn(\BD(X)).
\]

\begin{problemi}[Gelfand-type duality for coarse equivalences]
Let $X$ be a uniformly locally finite metric space. Is the canonical homomorphism
\[\Coa(X)\to \Out(\cst(X))\]
a group isomorphism?\label{ProblemB}
\end{problemi}

We give a positive answer to Problem \ref{ProblemB} above in the case of property A. This is proven as Theorem~\ref{ThmIsoCoa}.

\begin{theoremi}\label{thm:main}
Let $(X,d)$ be a uniformly locally finite metric space with property A. The canonical homomorphism
\[
\Coa(X)\to\Out(\cst(X))
\]
is a group isomorphism.
\end{theoremi}

Computing $\Coa(X)$ is in general a very difficult task, even for a simple space such as $\Z$. However, using results present in the literature, Theorem \ref{thm:main} gives us some interesting applications. For instance, $\Out(\cst(\Z))$ contains isomorphic copies of Thompson's group $F$ and of the free group of rank continuum (see Corollary \ref{CorThomp}). Using results of Eskin, Fisher and Whyte (\cite{EskinFisherWhyte2012Annals}), and of Farb and Mosher (\cite{FarbMosher1998Inventiones}), we obtain a complete computation of $\Out(\cst(X))$ for solvable Baumslag-Solitar groups, and for lamplighter graphs $F\wr \Z$, where $F$ is a finite group (see Corollaries~\ref{CorBaumslagSolitar} and~\ref{CorLamplighter}).

For our main results (Theorem~\ref{thm:uniform} and~\ref{thm:main}), we assume Yu's property A. This is one of the best known regularity properties in the setting of coarse spaces. It is equivalent to many algebraic and geometric properties such as the non-existence of noncompact ghost operators in $\cstu(X)$ (\cite[Theorem 1.3]{RoeWillett2014}), nuclearity of $\cstu(X)$ (\cite[Theorem 5.5.7]{BrownOzawa}), and the operator norm localisation property, ONL (\cite[Theorem 4.1]{Sako2014}) --- the latter is in fact the formulation of property A we use in our proofs. Property A is fairly broad: for example, all finitely generated exact groups have property A (more precisely, their Cayley graphs, when endowed with the shortest path metric, have property A). In particular, this includes the classes of linear groups, groups with finite asymptotic dimension and amenable groups.

 A key step in the proof of Theorem~\ref{thm:main} is to show `uniqueness of Cartan masas' in Roe algebras, generalising \cite[Theorem E]{WhiteWillett2017} to Roe algebras. In the case of uniform Roe algebras, the canonical Cartan masa of $\cstu(X)$ is $\ell_\infty(X)$. In $\cst(X)$, there are many canonical Cartan masas which depend on the choice of an orthonormal basis of $H$. Precisely, if $\bar\xi=(\xi_n)_n$ is an orthonormal basis of $H$, we obtain a Cartan masa of $\cst(X)$ by considering all operators $a\in \cB(\ell_2(X,H))$ such that for all $x\in X$ there is $(\lambda_n)_n\in c_0$ for which $a(\delta_x\otimes \xi_n)=\lambda_n(\delta_x\otimes \xi_n)$ for all $n\in\N$. This masa is isomorphic to $\ell_\infty(c_0)$ and we denote it by $\ell_\infty(X,\bar \xi)$.

 \begin{theoremi}\label{ThmCartan}
Let $X$ be a uniformly locally finite metric space with property A, and let $\bar\xi$ and $\bar\zeta$ be orthonormal bases of the separable Hilbert space $H$. If $\Phi\in\Aut(\cst(X))$, then there is a unitary $v\in\BD(X)$ such that \[\Ad(v)\circ\Phi(\ell_\infty(X,\bar \xi))=\ell_\infty(X,\bar \zeta).\]
\end{theoremi}

Theorem \ref{ThmCartan} is proven below as Theorem \ref{ThmIsoMakeAsympCoarseLike}. We point out that this result is actually stronger: the unitary $v\in \BD(X)$ is chosen so that $\Ad(v)\circ\Phi:\cst(X)\to \cst(X)$ respects the coarse geometry of $X$ in a very strong sense (see Definition \ref{DefiCoarseLikeProp} and Theorem \ref{ThmIsoMakeAsympCoarseLike} for details). 

To prove Theorem~\ref{ThmCartan}, and consequently Theorem~\ref{thm:main}, we state and prove uniform approximability results for Roe algebras. These results ensure that certain maps between Roe algebras respect in some sense the coarse geometry of the metric spaces. %More precisely, we prove that several maps (including isomorphisms) $\Phi:\cst(X)\to \cst(Y)$ send contractions with propagation at most $r$ to elements which are $\eps$ close in norm to elements with propagation at most $s=s(r,\eps)$ (see Definition \ref{DefiCoarseLikeProp} and Theorem \ref{ThmCoarseLike}). 
The concept of uniform approximability was introduced in \cite{BragaFarah2018}, and then further developed in \cite{BragaFarahVignati2019} and \cite{BragaFarahVignati2018}, for maps between uniform Roe algebras. These tools were already applied for a better understanding of uniform Roe algebras (e.g., \cite{LorentzWillett2020,WhiteWillett2017}), and we believe our generalisations can be the key to a further development of the theory of Roe algebras.
 
The paper is structured as follows: \S\ref{SecPrelim} contains preliminaries and sets the notation. Our uniform approximability results are proven in \S\ref{SecUnifApprox}, and applied in \S\ref{SecMult} and \S\ref{SectionMakingIsoCoarseLike}, where our main theorems are proved. \S\ref{SecApp} is dedicated to some applications.

\section{Notation and preliminaries}\label{SecPrelim}
If $H$ is a Hilbert space, $\cB(H)$ denotes the space of bounded operators on $H$ and $\cK(H)$ its ideal of compact operators. We denote the identity of $\mathcal B(H)$ by $1_H$. Given a set $X$, we denote by $\ell_2(X,H)$ the Hilbert space of square-summable functions $X\to H$. If $x\in X$ and $\eta\in H$, $\delta_x\otimes \eta$ is the function that sends $x$ to $\eta$ and all other elements of $X$ to $0$. Elements in $\mathcal B(\ell_2(X,H))$ can be viewed as (bounded) $X\times X$ matrices where each entry is an operator in $\mathcal B(H)$. With this identification, given $a\in \cB(\ell_2(X,H))$, and $x,y \in X$, we denote by $a_{xy}$ the operator in $\cB(H)$ determined by 
\[
\langle a_{xy}\xi,\eta\rangle=\langle a(\delta_x\otimes\xi),\delta_y\otimes\eta\rangle, \text{ for all } \xi,\eta\in H.\]
Multiplication in $\mathcal B(\ell_2(X,H))$ corresponds to matrix multiplication, that is,
\[
(ab)_{xy}=\sum_{z\in X}a_{xz}b_{zy}.
\]

We make the following abuse of notation throughout this article: given a metric space $X$ and a Hilbert space $H$, we write $\chi_{C}$ to denote both 
\begin{enumerate}
\item the projection in $\ell_\infty(X)\subset \cB(\ell_2(X))$ defined by $\chi_C\delta_x=\delta_x$ if $x\in C$ and $\chi_C\delta_x=0$ if $x\not\in C$, and
\item the projection in $\ell_\infty(X,\cB(H))$ defined by $\chi_C\delta_x\otimes \xi=\delta_x\otimes \xi$ if $x\in C$ and $\chi_C\delta_x\otimes \xi=0$ if $x\not\in C$. 
\end{enumerate}
If $\chi_C$ denotes a projection in $ \ell_\infty(X)$, when considering elements of $\cB(\ell_2(X,H))$, $\chi_C$ will always be accompanied by an operator on $H$ --- for instance: $\chi_C\otimes p$, where $p\in \cB(H)$. Given $a\in \cB(\ell_2(X,H))$ and $x,y\in X$, the operator $a_{xy}$ can be identified with the standard restriction of $\chi_{\{x\}}a\chi_{\{y\}}$ to an operator in $\cB(H)$. 

Let $(X,d)$ be a metric space. For $a\in\mathcal B(\ell_2(X,H))$, the \emph{support of $a$} is defined by $\supp(a)=\{(x,y)\in X^2\mid a_{xy}\neq 0\}$ and its \emph{propagation} by \[\propg(a)=\sup\{d(x,y)\mid (x,y)\in \supp(a)\}.\] If $H$ is infinite-dimensional and separable, we construct the \emph{band-dominated algebra of $X$}, denoted $\BD(X)$, by letting 
\[
\BD(X)=\overline{\{a\in\mathcal B(\ell_2(X,H))\mid \propg(a) \text{ is finite} \}}.
\]
The \emph{Roe algebra} is the ideal of $\BD(X)$ given by \emph{locally compact} operators i.e., $a_{xy}$ is compact for all $x,y\in X$. So 
\[
\cst(X)=\{a\in \BD(X) \mid \forall x,y\in X, \ a_{xy}\in\mathcal K(H)\}.
\]

If $H=\mathbb C$, all band-dominated operators are locally compact, as $\mathcal K(H)=\cB(H)$. In this case, the algebra of band-dominated operators is called the \emph{uniform Roe algebra of $X$} and denoted by $\cstu(X)$. 

The algebra $\cst(X)$ is not unital, but it has a well-behaved approximate identity if $X$ is uniformly locally finite. A metric space $(X,d)$ is \emph{uniformly locally finite} (\emph{u.l.f.}) if \[\sup_{x\in X}|B_r(x)|<\infty\ \text{ for all }\ r>0,\] where $|B_r(x)|$ denotes the cardinality of the closed ball of radius $r$ centered at $x$. Let $(p_j)_j$ be a sequence of finite rank projections on $H$ which converges to the identity $1_H$ in the strong operator topology. Given a function $f\colon X\to\N$, let \[q_f=\SOTh\sum_{x\in X}\chi_{\{x\}}\otimes p_{f(x)}.\] Each $q_f$ is in $\cst(X)$, and $q_f\leq q_g$ if and only if for all $x\in X$ we have $f(x)\leq g(x)$.
 
\begin{proposition}\label{prop:approxid}
Let $X$ be u.l.f.\@ metric space. The net $\{q_f\mid f\colon X\to\N\}$ is an approximate identity for $\cst(X)$ consisting of projections.
\end{proposition}
\begin{proof}
Pick $a\in\cst(X)$ with $\propg(a)\leq r$, and let $\e>0$. Since $X$ is u.l.f.\@ and $r$ is fixed, we can find $a_0$ and $a_1$ in $\cst(X)$ with propagation at most $ r$ and finite sets $A_n\subseteq X$, for $n\in\N$, such that
\begin{itemize}
\item $a=a_0+a_1$, 
\item $\supp(a_0)\cap \supp(a_1)=\emptyset$,
\item $a_0=\sum_{n }\chi_{A_{2n}}a_0\chi_{A_{2n}}$ and $a_1=\sum_{n }\chi_{A_{2n+1}}a_1\chi_{A_{2n+1}}$, and
\item if $|i-j|\geq 2$ then $A_i\cap A_j=\emptyset$.
\end{itemize}
Since each $A_i$ is finite and each entry of $\chi_{A_{2i}}a_0$ is compact, there is a sequence of natural numbers $(n(i))_i$ such that \[\norm{(\chi_{A_{2i}}\otimes p_{n(i)}) a_0-\chi_{A_{2i}}a_0}<\e\] for all $i\in\N$. Define $f_0\colon X\to \N$ by letting $f_0(x)=n(i)$ if $x\in A_{2i}$, and $f_0(x)=0$ elsewhere. Then $\norm{q_{f_0}a_0-a_0}<\e$. Notice that if $f_0\leq f$, then $q_fq_{f_0}=g_{f_0}$, hence 
\[
\norm{q_fa_0-a_0}\leq\norm{q_fa_0-q_fq_{f_0}a_0}+\norm{q_fq_{f_0}a_0-a_0}\leq2\e.
\]
One can analogously define $f_1\colon X\to X$ such that $\norm{q_{f_1}a_1-a_1}<\e$. Then $\norm{q_ga-a}\leq 4\e$
for all $g\colon X\to N$ with $g\geq {\max\{f_0,f_1\}}$.
Since $\e$ was arbitrary, we are done.
\end{proof}
By its definition, the Roe algebra $\cst(X)$ is an essential ideal in $\BD(X)$, hence $\BD(X)\subseteq\mathcal M(\cst(X))$ (see e.g., \cite[II.7.3.5]{Blackadar.OA}). We will prove that this is an equality in Theorem~\ref{prop:mult}.

%
% \begin{definition}
%Let $X$ be a metric space and $H$ be a separable Hilbert space. We denote the \cstar-algebra of all band-dominated operators on $\ell_2(X,H)$ by $\BD(X)$. 
%\end{definition}

%\begin{proposition}\label{PropInnBDInAut}
%Let $X$ be a metric space. Then 
%\[ \Inn(\BD(X))\subset \Aut(\cst(X)),\]
%i.e., if $u\in \BD(X)$ is unitary, then $\Ad(u)$ restricts to an automorphism of $\cst(X)$. 
%\end{proposition}

%\begin{proof}
% $\cst(X)$ is an ideal in $\BD(X)$, so $uau^*\in\cst(X)$ for all $a\in\cst(X)$ and $u\in \BD(X)$. Since $u$ is a unitary, the map $a\mapsto uau^*$ is injective and surjective on $\cst(X)$.
%\end{proof}

%Let $v\in \BD(X)$ be an element with finite propagation. Let $a\in \cst(X)$ and $x,z\in X$. Then, as $X$ is locally finite and $|propg(v)<\infty$, there are finite $A,B\subset X$ so that $\chi_{\{x\}}v=\chi_{\{x\}}v\chi_A$ and $v^*\chi_{\{z\}}=\chi_{B}v^*\chi_{\{z\}}$. So
%\[\chi_{\{x\}}vav^*\chi_{\{z\}}=\chi_{\{x\}}v\chi_Aa\chi_Bv^*\chi_{\{z\}}.\]
%As $a$ is locally compact and both $A$ and $B$ are finite, 
%$\chi_Aa\chi_B$ is compact; hence so is $\chi_{\{x\}}vav^*\chi_{\{z\}}$. As any element in $ \BD(X)$ is the norm limit of band operators and as the norm limit of a compact operator is still compact, the result follows. 
%\end{proof}

\subsection{Coarse geometry}
In the setting of coarse geometry, homeomorphisms are replaced by maps preserving the large-scale geometry. This is a crucial definition we will be using in the whole paper.
\begin{definition}\label{def:close}
Let $(X,d)$ and $(Y,\partial)$ be metric spaces, and let $f$ and $g$ be functions $X\to Y$. The map $f$ is called \emph{coarse} if for all $r>0$ there is $s>0$ so that 
\[d(x,y)<r\text{ implies }\partial (f(x),f(y))<s.\]
We say that $f$ and $g$ are \emph{close} if
\[
\sup_x\partial(f(x),g(x))<\infty.
\]
The map $f$ is called a \emph{coarse equivalence} if it is coarse and there exists a coarse map $h\colon Y\to X$ so that $f\circ h$ and $h\circ f$ are close to $\mathrm{Id}_Y$ and $\mathrm{Id}_X$, respectively.
\end{definition}

Notice that a coarse equivalence $f\colon X\to Y$ is automatically \emph{cobounded}, i.e., $\sup_{y\in Y}\partial(y,f(X))<\infty$.

Throughout the paper, every time $X$ and $Y$ are metric spaces, we will assume without further notice that $d$ and $\partial$ are the metrics of $X$ and $Y$, respectively. 

\subsection{The canonical maps}\label{SubsectionCanonicalMap}
We now present the construction of the canonical map associating an automorphism of $\cstu(X)$ to a bijective coarse equivalence of $X$. We then prove Theorem~\ref{thm:uniform}. Finally, we associate an automorphism of $\cst(X)$ to a coarse equivalence of $X$. Although such an association is highly noncanonical we prove that it becomes canonical when $\Aut(\cst(X))$ is replaced by $\Out(\cst(X))$.

Let $f\colon X\to X$ be a bijective coarse equivalence. Consider the unitary $v_f\in \cB(\ell_2(X))$ given by $v_f\delta_x=\delta_{f(x)}$ for all $x\in X$. Since $f$ and its inverse are coarse, $\Ad(v_f)$ is an automorphism of $\cstu(X)$. %\footnote{Given a unitary $v$ in a \cstar-algebra $A$, $\Ad(v):A\to A$ is the isomorphism given by $\Ad(v)(a)=vav^*$ for all $a\in A$.} 
If $f$ and $g$ are bijections, then $v_{f\circ g}=v_fv_g$ and $v_{f^{-1}}=v_f^*$. Moreover, $v_f\in\cstu(X)$ if and only if $f$ is close to the identity; this gives an injective homomorphism $\BijCoa(X)\to\Out(\cstu(X))$. The proof of Theorem~\ref{thm:uniform} amounts then to show surjectivity in case property A is assumed. To do so, we recall the work of White and Willett, who proved uniqueness of Cartan masas in uniform Roe algebras for property A spaces. We slightly restate their result.
\begin{theorem}[Theorem E of \cite{WhiteWillett2017}]\label{thm:WW}
Let $X$ be a u.l.f.\@ metric space with property $A$. Let $\Phi\in\Aut(\cstu(X))$. Then there is a unitary $u\in\cstu(X)$ such that $\Ad(u)\circ\Phi(\ell_\infty(X))=\ell_\infty(X)$.
\end{theorem}
We restate Theorem~\ref{thm:uniform} for convenience.
\begin{theorem}\label{ThmIsoCoarseAutURA}
Let $X$ be a u.l.f.\@ metric space with property A. The canonical map 
\[
\BijCoa(X)\to \Out(\cstu(X))
\]
is a group isomorphism.
\end{theorem}
\begin{proof}
If $\Phi \in \Aut(\cstu(X))$, then Theorem \ref{thm:WW} gives a unitary $u \in \cstu(X)$ so that $\Psi=\Ad(u)\circ \Phi$ is an automorphism of $\cstu(X)$ which takes $\ell_\infty(X)$ to itself. As every automorphism of $\cstu(X)$ is implemented by a unitary in $\cB(\ell_2(X))$ (see \cite[Lemma 3.1]{SpakulaWillett2013AdvMath}), let $v$ be this unitary, i.e., $\Psi=\Ad(v)$. As $v\ell_\infty(X)v*=\ell_\infty(X)$, there is a bijection $f\colon X\to X$ and a family $(\lambda_x)_{x\in X}$ in the unit circle of $\mathbb C$ so that $v\delta_x=\lambda_x\delta_{f(x)}$ for all $x\in X$ (see \cite[Lemma 8.10]{BragaFarah2018} for a proof of that). Hence, $\Ad(v_f)$ equals $\Psi$ modulo $\Inn(\cstu(X))$, which in turn, as $u\in \cstu(X)$, equals $\Phi$ modulo $\Inn(\cstu(X))$. 
\end{proof}
We now present a map which associates to a coarse equivalence of $X$ an automorphism of $\cst(X)$. This construction is well known to specialists but, as we use its specifics in the proof of Theorem~\ref{thm:main}, we give its details in full.

Fix metric spaces $X$ and $Y$, two orthonormal bases of $H$, $\bar\xi=(\xi_n)_n$ and $\bar\zeta=(\zeta_n)_n$, and let $f\colon X\to Y$ be a coarse equivalence. Let $Y_0=f[X]$. For each $y\in Y_0$, pick $x_y\in X$ such that $f(x_y)=y$, and let $X_0=\{x_y\mid y\in Y_0\}$. Since $f$ is 	a coarse equivalence, $X_0$ and $Y_0$ are cobounded in $X$ and $Y$, respectively. By the last statement, we can pick sequences of disjoint finite sets $(X^x)_{x\in X_0}$ and $(Y^y)_{y\in Y_0}$, and $r_0>0$, so that 
\begin{enumerate}
\item $X=\bigsqcup_{x\in X_0}X^x$ and $Y=\bigsqcup_{x\in Y_0}Y^y$,
\item $x\in X^x$ and $\diam(X^x)\leq r_0$ for all $x\in X_0$, and
\item $y\in Y^y$ and $\diam(Y^y)\leq r_0$ for all all $y\in Y_0$.
\end{enumerate} 
For each $x\in X_0$, pick a bijection \[g_x\colon X^x\times \N\to Y^{f(x)}\times \N.\] Define
\[
g=\bigcup_{x\in X_0}g_x.
\] So $g$ is a bijection between $X\times\N$ and $Y\times\N$. Let $g_1\colon X\times \N\to Y$ and $g_2\colon X\times \N\to \N$ be the compositions of $g$ with the projections onto the first and second coordinates, respectively. If $x\in X$, we write $g_1(x,\N)$ for the set $\{g_1(x,n)\mid n\in\N\}$.

 Define a unitary $u=u_g\colon\ell_2(X,H)\to \ell_2(Y,H)$ by
\[
u\delta_x\otimes \xi_n=\delta_{g_1(x,n)}\otimes \zeta_{g_2(x,n)}
\]
for all $(x,n)\in X\times \N$. For $a\in\mathcal B(\ell_2(X,H))$, define $\Psi\colon\cB(\ell_2(X,H))\to \cB(\ell_2(Y,H))$ by
\[
\Psi(a)=uau^*.
\]
Notice that $\Psi$ maps locally compact elements to locally compact elements. We intend to show that the image of $\cst(X)$ is in $\cst(Y)$. 

\begin{claim}\label{Claim1}
If $a$ has finite propagation so does $\Psi(a)$.
\end{claim}

\begin{proof}
Fix $r>0$ and pick $s'$ such that if $d(z,z')\leq r+2r_0$ then $\partial(f(z),f(z'))\leq s'$ whenever $z$ and $z'$ are in $X$. This exists as $f$ is coarse. 

Suppose now that $x$ and $x'$ are elements of $X$ with $d(x,x')\leq r$. Let $y$ and $y'$ be such that $y\in g_1(x,\N)$ and $y'\in g_1(x',\N)$. Pick $z$ and $z'$ such that $x\in X^z$ and $x'\in X^{z'}$. Notice that $y\in Y^{f(z)}$ and $y'\in Y^{f(z')}$. Since $f$ is coarse and the diameter of each $X^z$ is bounded by $r_0$, $d(z,z')\leq r+2r_0$, hence $\partial(f(z),f(z'))\leq s'$. Since the diameter of each $Y^y$ is bounded by $r_0$, we have that $\partial(y,y')\leq s'+2r_0$. 

Pick now $a\in\BD(X)$ with $\propg(a)\leq r$, and suppose that $y$ and $y'$ are elements on $Y$ with $\partial(y,y')> s'+2r_0$. Fix also $n,n'\in \N$. Let $w$ and $w'$ be in $X$ and, let $m,m'\in \N$ be such that $g(w,m)=(y,n)$ and $g(w',m')=(y',n')$. Since $\partial(y,y')>s'+2r_0$, then $d(w,w')>r$, hence $a_{ww'}=0$. In particular 
\begin{align*}
0&=\langle a(\delta_w\otimes\xi_n), \delta_{w'}\otimes\xi_{n'}\rangle=\langle u^*\Psi(a)u(\delta_w\otimes\xi_n), \delta_{w'}\otimes\xi_{n'}\rangle\\&=\langle \Psi(a)u(\delta_w\otimes\xi_n), u\delta_{w'}\otimes\xi_{n'}\rangle=\langle \Psi(a)(\delta_y\otimes\zeta_m),\delta_{y'}\otimes\zeta_{m'}\rangle
\end{align*}
Since $m$ and $m'$ are arbitrary, then $\Psi_{yy'}=0$. Since $y$ and $y'$ were arbitrary elements such that $\partial(y,y')>s'+2r_0$, then $\propg(\Psi(a))\leq s'+2r_0$. This concludes the proof.
\end{proof}
Claim~\ref{Claim1} implies that $\Psi(\cst(X))\subset \cst(Y)$. By symmetry of the arguments, it follows that $\Psi^{-1}(\cst(Y))\subset \cst(X)$. Hence $\Psi$ restricts to an isomorphism between $\cst(X)$ and $\cst(Y)$. We set 
\[
\Phi_{f,(\xi_n),(\zeta_n),g}=\Psi\rs\cst(X).
\]
This map is highly noncanonical, as it depends on the choices of $\bar\xi$, $\bar\zeta$, and $g$. (The latter in turns depends on $f$, but not canonically.) We want to make this association canonical.

If $\Phi$ and $\Psi$ are isomorphisms between $\cst(X)$ and $\cst(Y)$, we write
\[
\Phi\sim_{u,\BD}\Psi \iff\exists u\in \BD(Y) \text{ such that } \Phi=\Ad(u)\circ\Psi.
\]
We show that our association becomes canonical when $\Aut(\cst(X))$ is divided by $\Inn(\BD(X))$.
\begin{proposition}
Let $X$ and $Y$ be u.l.f.\@ metric spaces, and let $f$ and $h$ be coarse equivalences between $X$ and $Y$. Suppose that $\Phi_{f,(\xi_n),(\zeta_n),g}$ and $\Phi_{h,(\xi'_n),(\zeta'_n),g'}$ are constructed as above from different parameters. Then 
\[
\Phi_{f,(\xi_n),(\zeta_n),g}\sim_{u,\BD}\Phi_{h,(\xi'_n),(\zeta'_n),g'} \text{ if and only if } f \text{ is close to } h.
\]
\end{proposition} 
 
\begin{proof}
Let $u\in\BD(X)$ be the unitary such that for all $x\in X$ and $n\in\NN$, $u\delta_x\otimes\xi_n=\delta_x\otimes\xi'_n$, and let $w\in\BD(Y)$ be the unitary such that for all $y\in Y$ and $n\in\N$, $w(\delta_y\otimes \zeta_n)=\delta_y\otimes\zeta'_n$. Then 
\[
\Phi_{h,(\xi_n),(\zeta_n),g'}=\Ad(w^*)\circ \Phi_{h,(\xi'_n),(\zeta'_n),g'}\circ\Ad(u).
\]
Since $w$ and $\Phi_{h,(\xi'_n),(\zeta'_n),g'}(u)$ are in $\BD(Y)$, then 
\[
\Phi_{f,(\xi_n),(\zeta_n),g}\sim_{u,\BD}\Phi_{h,(\xi'_n),(\zeta'_n),g'} \ \Leftrightarrow\ \Phi_{f,(\xi_n),(\zeta_n),g}\sim_{u,\BD}\Phi_{h,(\xi_n),(\zeta_n),g'}.
\]
Let $g''\colon Y\times\N\to Y\times\N$ given by $g''=g\circ g'^{-1}$. Since $g$ and $g'$ are bijections, so is $g''$. Define a unitary $v\in\mathcal B(\ell_2(Y,H))$ by 
\[
v(\delta_y\otimes\zeta_n)=\delta_{g''_1(y,n)}\otimes\zeta_{g''_2(y,n)}.
\]
It is routine to check that $\Phi_{f,(\xi_n),(\zeta_n),g}=\Ad(v)\circ\Phi_{h,(\xi'_n),(\zeta'_n),g'}$. Hence we just need to show that $v\in\BD(Y)$. This follows from how $g$ and $g'$ are constructed because $f$ and $h$ are close.
\end{proof}

We have constructed a canonical injective homomorphism $\Coa(X)\to\Out(\cst(X))$. As in the proof of Theorem~\ref{thm:uniform}, our efforts will amount to prove surjectivity.

\section{Uniform approximability in Roe algebras}\label{SecUnifApprox}

This section deals with uniform approximability of maps $\Phi$ between \cstar-subalgebras of $\cB(\ell_2(X,H))$ and $\cB(\ell_2(Y,H))$. Precisely, in this section we study when maps satisfy coarse-like properties, that is, when morphisms respect the large-scale geometry of the underlying spaces. This concept was introduced for maps between uniform Roe algebras in \cite{BragaFarah2018}, and formalised in \cite{BragaFarahVignati2018}; here we define it in the setting of Roe algebras.

 \begin{definition}\label{DefiCoarseLikeProp}
Let $X$ and $Y$ be metric spaces, $\cA\subset \cB(\ell_2(X,H))$ and $\cB\subset \cB(\ell_2(Y,H))$ be \cstar-subalgebras. 
\begin{enumerate}
\item Given $\eps,r>0$, an element $a\in \cA$ is \emph{$\eps$-$r$-approximable in $\cA$} if there is $c\in \cA$ with $\propg(c)\leq r$ so that $\|a-c\|\leq \eps$.
\item A map $\Phi\colon\cA\to \cB$ is \emph{coarse-like} if for all $\eps,r>0$ there is $s>0$ so that $\Phi(a)$ is $\eps$-$s$-approximable in $\cB$ for all contractions in $a\in \cA$ with $\propg(a)\leq r$. 
\end{enumerate} 
\end{definition} 

%The next proposition states some properties of these notions; since their proofs are straightforward, we omit them.

%\begin{proposition}\label{proppreservation}
%Let $X$, $Y$, and $Z$ be u.l.f.\@ metric spaces. 
%\begin{itemize}
%\item If $f\colon X\to Y$ is a coarse equivalence, the map $\Phi_f\colon \cst(X)\to\cst(Y)$ given in \S\ref{SubsectionCanonicalMap} and its inverse are strongly coarse-like;
%\item If $\cA\subseteq\cB(\ell_2(X,H))$, $\cB\subset \cB(\ell_2(Y,H))$ and $\cC\subseteq\cB(\ell_2(Z,H))$ are subalgebras, and $\Phi\colon\cA\to\cB$ and $\Psi\colon\cB\to\cC$ are (strongly) coarse-like, then $\Psi\circ\Phi$ is (strong) coarse-like.\qed
%\end{itemize}
%\end{proposition}

The following theorem is the starting point for our research on uniform approximability for Roe algebras. It is a simple consequence of \cite[Lemma 4.9]{BragaFarah2018} (see \cite[Proposition 3.3]{BragaFarahVignati2019} for a precise proof; cf.\@ \cite[Theorem 4.4]{BragaFarah2018}). 

\begin{theorem}\label{ThmCoarseLikeURA}
Let $X$ and $Y$ be u.l.f.\@ metric spaces and let $\Phi\colon\cstu(X)\to \cstu(Y)$ be a strongly continuous compact preserving linear map. Then $\Phi$ is coarse-like. \qed
\end{theorem}
 
In layman terms, Theorem \ref{ThmCoarseLikeURA} says that, for uniform Roe algebras, uniform approximability holds in a very strong sense. This result is false for Roe algebras.

\begin{proposition}\label{PropNotCoarseLike}
Given a finite metric space $X$ and a metric space of infinite diameter $Y$, there is a compact preserving strongly continuous embedding $\Phi\colon\cst(X)\to \cst(Y)$ onto a hereditary subalgebra of $\cst(Y)$ which is not coarse-like. 
\end{proposition}

\begin{proof}
For simplicity, we assume $X$ to be a singleton, say $X=\{x\}$, and $Y$ to be countable, say $Y=\N$. Fix a metric $\partial$ on $Y$ of infinite diameter. Let $(\xi_n)_n$ be an orthonormal basis for $H$, and define $u\colon\ell_2(X,H)\to \ell_2(Y,H)$ by 
\[
u\delta_x\otimes \xi_n=\delta_n\otimes \xi_1, \,\,\text{ for } n\in\N.
\]
The operator $u$ is an isometry, and $\Ad(u)$ is a compact preserving strongly continuous embedding of $\mathcal B(\ell_2(X,H))$ into $\mathcal B(\ell_2(Y,H))$. As $\cst(X)=\chi_{\{x\}}\otimes \cK(H)$, then $\Ad(u)(\cst(X))\subset \cK(\ell_2(Y,H))\subset \cst(Y)$. The image of $\Ad(u)$ is a hereditary subalgebra of $\cst(Y)$ since it equals $\cK(\ell_2(Y))\otimes q_1$, where $q_1\colon H\to\Span \{\xi_1\}$ is the standard projection.

We are left to show that $\Ad(u)$ is not coarse-like. Fix $n\in\N$, and let $m\in\N$ be such that $\partial(1,m)\geq n$. Let $v\in\cB(\ell_2(X,H))$ be the rank $1$ partial isometry sending $\delta_x\otimes \xi_1$ to $\delta_x\otimes \xi_m$. Then $v$ has propagation $0$ but, as \[\langle \Ad(u)v(\delta_1\otimes \xi_1),\delta_m\otimes \xi_1\rangle=1,\] it follows that $\Ad(u)v$ cannot be $1/2$-$n$-approximated. As $n$ is arbitrary, we are done. 
\end{proof}

The map of Proposition~\ref{PropNotCoarseLike} sends a sequence which is converging in the strong operator topology to an element outside of $\cst(X)$ to a sequence which is converging in the strong operator topology to an element in $\cst(Y)$. This is the only obstacle in generalising Theorem~\ref{ThmCoarseLikeURA} to Roe algebras.

The following two theorems are our main uniform approximability results and most of this section is dedicated to prove them.

 \begin{theorem}\label{ThmCoarseLike}
 Let $X$ and $Y$ be u.l.f.\@ metric spaces. Then every isomorphism between $\cst(X)$ and $\cst(Y)$ is coarse-like.\end{theorem}

Although Proposition \ref{PropNotCoarseLike} shows that Theorem \ref{ThmCoarseLikeURA} cannot be extended to Roe algebras, the latter can be extended to band-dominated algebras as follows:

\begin{theorem}\label{ThmCoarseLikeBD}
Let $X$ and $Y$ be u.l.f.\@ metric spaces. Then every strongly continuous compact preserving $^*$-homomorphism $\Phi\colon \BD(X)\to \BD(Y)$ is coarse-like.
\end{theorem}

We now proceed to prove Theorems~\ref{ThmCoarseLike} and \ref{ThmCoarseLikeBD}. First, we show that $\e$-$r$-approximability does not depend on the ambient algebra.
\begin{proposition}\label{PropApprox}
 Let $X$ be a u.l.f.\@ metric space, $r>0$, $\e>0$, and $a\in \cst(X)$. The following are equivalent: 
 \begin{enumerate}
 \item\label{ItemPropApprox1} $a$ is $\e$-$r$-approximable in $\BD(X)$, 
 \item\label{ItemPropApprox2} $a$ is $(\e+\delta)$-$r$-approximable in $\cst(X)$, for all $\delta>0$, and
 \item\label{ItemPropApprox3} $a$ is $(\e+\delta)$-$r$-approximable in $\BD(X)$, for all $\delta>0$.
 \end{enumerate}
\end{proposition} 
 \begin{proof}
 \eqref{ItemPropApprox1}$\Rightarrow$\eqref{ItemPropApprox2} If $a$ is $\e$-$r$-approximable in $\BD(X)$, pick $b$ of propagation $r$ with $\norm{a-b}\leq\e$. Let $p\in\cst(X)$ be a projection with $\propg(p)=0$ and such that $\norm{pa-a}<\delta$. This exists by Proposition~\ref{prop:approxid}. Then $pb\in\cst(X)$, $\propg(pb)\leq r$ and \[\norm{a-pb}\leq \norm{a-pa+pa-pb}<\norm{a-b}+\delta/2<\e+\delta.\] 
 
 \eqref{ItemPropApprox2}$\Rightarrow$\eqref{ItemPropApprox3} This is immediate. 
 
 \eqref{ItemPropApprox3}$\Rightarrow$\eqref{ItemPropApprox1} For each $n\in\N$, pick $b_n\in \BD(X)$ with $\propg(b_n)\leq r$ and $\|a-b_n\|\leq \eps+1/n$. Then $(b_n)_n$ is bounded and, by compactness of the unit ball of $\cB(\ell_2(X,H))$ with respect to the weak operator topology, by going to a subsequence if necessary, we can define $b=\WOTh\lim_n b_n$. Clearly, $\propg(b)\leq r$, so we only need to notice that $\|a-b\|\leq \eps$. Suppose $\|a-b\|>\eps$. Pick unit vectors $\xi$ and $\zeta$ in $\ell_2(X,H)$, and $n\in\N$ so that $|\langle (a-b)\xi,\zeta\rangle|>\eps+1/n$. By the definition of $b$, there is $m>n$ so that $|\langle (a-b_m)\xi,\zeta\rangle|>\eps+1/n$. As $m>n$, this implies that $\|a-b_m\|> \eps+1/n$; contradiction.
 \end{proof}

The set of $\e$-$r$-approximable elements is strongly closed:
 
 \begin{proposition}\label{PropSOTConvApproxBD}
 Let $X$ be a u.l.f.\@ metric space, $r>0$, $\eps>0$, $a\in \cB(\ell_2(X,H))$ and let $(a_n)_n$ be a sequence in $\BD(X)$ so that $a=\SOTh\lim a_n$. If each $a_n$ is $\eps$-$r$-approximable in $\BD(X)$, then $a$ is $\eps$-$r$-approximable in $\BD(X)$. 
 \end{proposition}

\begin{proof}
As each $a_n$ is $\eps$-$r$-approximable in $\BD(X)$, pick a sequence $(b_n)_n$ in $\BD(X)$ so that $\propg(b_n)\leq r$ and $\|a_n-b_n\|\leq \eps$ for all $n\in\N$. As $a=\SOTh\lim_na_n$, the principle of uniform boundedness implies that $(a_n)_n$ is a bounded sequence, hence so is $(b_n)_n$. By compactness of the unit ball of $\cB(\ell_2(X,H))$ in the weak operator topology, going to a subsequence if necessary, we can assume that $b=\WOTh\lim_nb_n$ exists. Clearly, $\propg(b)\leq r$ and $\|a-b\|\leq \eps$, so we are done. 
 \end{proof}

Propositions~\ref{PropApprox} and~\ref{PropSOTConvApproxBD} together give the following:

\begin{proposition}\label{PropSOTConvApprox}
 Let $X$ be a u.l.f.\@ metric space, $r>0$, $\eps>0$, $a\in\cst(X)$ and let $(a_n)_n$ be a sequence in $\BD(X)$ so that $a=\SOTh\lim a_n$. If each $a_n$ is $\eps$-$r$-approximable in $\cst(X)$, then $a$ is $(\eps+\delta)$-$r$-approximable in $\cst(X)$ for all $\delta>0$. \qed
 \end{proposition}

Next, we study how strong convergence and $\e$-$r$-approximability interact. We prove the Roe algebra and the band-dominated algebra versions of \cite[Lemma 4.9]{BragaFarah2018}. Let $\D=\{z\in \C\mid |z|\leq 1\}$. If $(a_n)_n$ is a sequence of operators and $\bar\lambda\in\D^\N$ is such that $\SOTh\sum_n\lambda_na_n$ exists, we write 
\[
a_{\bar\lambda}=\SOTh\sum_n\lambda_na_n.
\]
When writing $a_{\bar\lambda}$, it is implicit that the limit above exists.
 \begin{lemma}\label{Lemma4.9BD}
 Let $X$ be a u.l.f.\@ metric space and let $(a_n)_n$ be a sequence of compact operators in $\BD(X)$ so that $a_{\bar\lambda}\in \BD(X)$ for all $\bar\lambda\in \D^\N$. Then for all $\eps>0$ there is $r>0$ so that $a_{\bar\lambda}$ is $\eps$-$r$-approximable in $\BD(X)$ for all $\bar\lambda\in \D^\N$.
 \end{lemma}

\begin{proof}
 Let $(p_j)_j$ be a sequence of finite rank projections on $H$ which strongly converges to $1_H$. Let $1_X$ denote the identity on $\ell_2(X)$ and notice that $a=\SOTh \lim_j(1_X\otimes p_j) a (1_X\otimes p_j)$. Given $\bar\lambda=(\lambda_n)_n\in \D^\N$ and $j\in\N$, let \[a_{\bar\lambda,j}= (1_X\otimes p_j)a_{\bar\lambda} (1_X\otimes p_j),
 \]
 %\ \ \Big(\text{i.e., }a_{\bar\lambda,j} =\SOTh\sum_n\lambda_n 1_X\otimes p_ja_n1_X\otimes p_j\Big) \]
and $a_{\bar\lambda,\infty}=a_{\bar\lambda}$, so that $a_{\bar\lambda,\infty}=\SOTh\lim_j a_{\bar\lambda,j}$ for all $\bar\lambda\in \D^\N$. By Proposition \ref{PropSOTConvApprox}, it is enough to show that
 \begin{itemize}
 \item[$(\ast)$]\label{ItemStatement} for all $\eps>0$ there is $r>0$ so that $a_{\bar\lambda,j} $ is $\eps$-$r$-approximable in $\BD(X)$ for all $\bar\lambda\in \D^\N$ and all $j\in\N$. 
 \end{itemize}
Suppose ($\ast$) fails for $\eps>0$. For each finite $I\subset \N$, let
 \[\cZ_I=\Big\{\bar\lambda\in \D^\N\mid \forall j\in I, \lambda_j=0\Big\}\text{ and }\cY_I=\Big\{\bar\lambda\in \D^\N\mid \forall j\not\in I, \lambda_j=0\Big\}.\]
So $\cY_I$ is compact in the product space $\D^\N$.

\begin{claim}\label{Claim1Lemma4.9}
 For all $r>0$ and all finite $I\subset \N$, there exist $\bar\lambda\in \cZ_I$ and $j\in\N$ so that $a_{\bar\lambda,j}$ is not $\eps/2$-$r$-approximable in $\BD(X)$.
\end{claim} 

\begin{proof}
Suppose the claim fails and let $r>0$ and $I\subset \N$ witness that. Let $\N_\infty$ be the one-point compactification of $\N$. Notice that the map 
\[(\bar\lambda,j)\in \cY_I\times \N_\infty\mapsto a_{\bar\lambda,j}\in \BD(X)\]
is continuous. Indeed, the map is clearly continuous on $\cY_I\times \N$ and continuity on $\cY_I\times \{\infty\}$ follows since each $a_n$ is compact; therefore, $a_n=\lim_j(1_X\otimes p_j)a_n(1_X\otimes p_j)$ for all $n\in\N$ and it follows that $a_{\bar\lambda,\infty}=\lim_ja_{\bar\lambda, j}$ for all $\bar\lambda\in \cY_I$. The continuity of this map and the compactness of $\cY_I\times \N_\infty$ imply that 
 $\{a_{\bar\lambda,j}\mid \bar\lambda\in \cY_I, j\in\N_\infty\}$ is a norm compact subset of $\BD(X)$. In particular, the total boundedness of this set gives $s>0$ so that $a$ is $\eps/2$-$s$-approximable in $\BD(X)$ for all $a\in 
 \{a_{\bar\lambda,j}\mid \bar\lambda\in \cY_I, j\in\N\}$.

Let $\bar\lambda\in \D^\N$ and $j\in\N$. Write $\bar\lambda=\bar\lambda_1+\bar\lambda_2$ for $\bar\lambda_1\in \cY_I$ and $\bar\lambda_2\in 
 \cZ_I$, so $a_{\bar\lambda,j}=a_{\bar\lambda_1,j}+a_{\bar\lambda_2,j}$. As $r$ and $I$ witness that the claim fails, $a_{\bar\lambda_2,j}$ is $\eps/2$-$r$-approximable in $\BD(X)$. Moreover, our choice of $s$ implies that $a_{\bar\lambda_1,j}$ is $\eps/2$-$s$-approximable in $\BD(X)$; hence $a_{\bar\lambda,j}$ is $\eps$-$\max\{r,s\}$-approximable in $\BD(X)$. As $\bar\lambda\in \D^\N$ and $j\in\N$ were arbitrary, this contradicts that $(\ast)$ fails for $\eps$.
\end{proof}
 
 For $r>0$ and $\delta>0$, let \[U_{r,\delta}=\Big\{\bar\lambda \in \D^\N \mid a_{\bar\lambda,j} \text{ is }\delta\text{-}r\text{-approximable in }\BD(X)\text{ for all }j\in\N\Big\}.\] 
 
 \begin{claim}\label{claim:NWD1}
The set $U_{r,\delta}$ is closed for all $r>0$ and $\delta>0$. 
 \end{claim}
 
 \begin{proof}
Suppose the claim fails for $r>0$ and $\delta>0$. Then there is $\bar\lambda \in U^\complement_{r,\delta}\cap \overline{U_{r,\delta}}$. As $\bar\lambda \not\in U_{r,\delta}$, there is $j\in\N$ so that $a_{\bar\lambda,j}$ is not $\delta$-$r$-approximable in $\BD(X)$. Proposition \ref{PropSOTConvApprox} implies that there is a finite $F\subset X$ so that $\chi_{F}a_{\bar\lambda,j}\chi_{F} $ is not $\delta$-$r$-approximable in $\BD(X)$. 

Fix $\gamma>0$. By the definition of $a_{\bar \theta,j}$, $\chi_{F}a_{\bar\theta,j}\chi_{F}=(\chi_{F}\otimes p_j) a_{\bar\theta}(\chi_{F}\otimes p_j)$; since  $\chi_{F}\otimes p_j$ is compact, then there exists a finite $I\subset \N$ so that $\|\chi_{F}a_{\bar\theta,j}\chi_{F}\|<\gamma$ for all $\bar\theta\in \cZ_I$. Indeed, this follows since 
\[\chi_{F}a_{\bar\theta,j}\chi_{F}=\SOTh\sum_n\theta_n(\chi_{F}\otimes p_j)a_n (\chi_{F}\otimes p_j) \]
for all $\bar\theta\in \D^\N$. Let $\bar\lambda_1\in \cY_I$ and $\bar\lambda_2\in \cZ_I$ be so that $\bar\lambda=\bar\lambda_1+\bar\lambda_2$. As $\bar\lambda\in \overline{U_{r,\delta}}$, there exists $\bar\theta_1\in \cY_I$ and $\bar\theta_2\in \cZ_I$ so that $\bar\theta=\bar\theta_1+\bar\theta_2\in U_{r,\delta}$ and $\|a_{\bar\lambda_1,j}-a_{\bar\theta_1,j}\|\leq \gamma$. As $\bar\theta\in U_{r,\delta}$, $a_{\bar\theta,j}$ is $\delta$-$r$-approximable in $\BD(X)$. In particular, as $\chi_{F}\otimes p_m$ is a projection with propagation $0$, $\chi_{F}a_{\bar\theta,j}\chi_{F}$ is $\delta$-$r$-approximable in $\BD(X)$. Therefore, since 
\[a_{\bar\lambda,j}=a_{\bar\theta,j}+a_{\bar\lambda_1,j}-a_{\bar\theta_1,j}+a_{\bar\lambda_2,j}-a_{\bar\theta_2,j},\]
then $\chi_{F}a_{\bar\lambda,j}\chi_{F}$ is $(\delta+3\gamma)$-$r$-approximable in $\BD(X)$. As $\gamma$ was arbitrary, Proposition \ref{PropApprox} implies that $\chi_{F}a_{\bar\lambda,j}\chi_{F}$ is $\delta $-$r$-approximable in $\BD(X)$; contradiction.
 \end{proof}
Fix $\delta=\eps/4$. 
 
\begin{claim}\label{claim:NWD2}
 For all $r>0$, the set $U_{r,\delta}$ has empty interior. 
\end{claim} 
 
\begin{proof}
Fix $r>0$ and let $\bar\lambda\in U_{r,\delta}$. Fix a finite $I\subset \N$ and write $\bar\lambda=\bar\lambda_1+\bar\lambda_2$, for $\bar\lambda_1\in \cY_I$ and $\bar\lambda_2\in \cZ_I$. Pick $s>r$ so that $a_{\bar\lambda_1, j}$ is $\delta$-$s$-approximable in $\BD(X)$. By Claim \ref{Claim1Lemma4.9}, there is $\bar\theta_2\in \cZ_I$ and $j\in\N$ so that $a_{\bar\theta_2,j}$ is not $2\delta$-$s$-approximable in $\BD(X)$. Hence, letting $\bar\theta=\bar\lambda_1+\bar\theta_2$, we have that $a_{\bar\theta,j}$ is not $\delta$-$s$-approximable in $\BD(X)$. As $s>r$, $a_{\bar\theta,j}$ is not $\delta$-$r$-approximable in $\BD(X)$, so $\bar\theta\not\in U_{r,\delta}$. Since $I$ was an arbitrary finite subset of $\N$, this shows that $\bar\lambda$ is not an interior point of $U_{r,\delta}$.
\end{proof} 
By Claim~\ref{claim:NWD1} and \ref{claim:NWD2}, $U_{r,\delta}$ is nowhere dense for all $r>0$. However
\[\D^\N=\bigcup_{n\in\N}U_{n,\delta}.\]
Indeed, if $\bar\lambda\in \D^\N$, then there is $n\in\N$ so that $a_{\bar\lambda,\infty}$ is $\delta$-$n$-approximable in $\BD(X)$. In particular, $a_{\bar\lambda,j}$ is $\delta$-$n$-approximable in $\BD(X)$ for all $j\in\N$, so $\bar\lambda\in U_{n,\delta}$. Since $\D^\N$ is a Baire space, we have a contradiction.
\end{proof}

\begin{remark}
The Baire categorical nature of the proof of Lemma \ref{Lemma4.9BD} implies that its statement holds outside the scope of metrizable coarse spaces (for brevity, we refer the reader to \cite{BragaFarah2018} for definitions). Indeed, if $(X,\cE)$ is a coarse space which is \emph{small} (see \cite[Definition 4.2]{BragaFarah2018}), Lemma \ref{Lemma4.9BD} still holds. In particular, Theorems \ref{ThmCoarseLike} and \ref{ThmCoarseLikeBD} also hold for small coarse spaces.
\end{remark}

 The following is the Roe algebra version of Lemma~\ref{Lemma4.9BD}.
 
 \begin{lemma}\label{Lemma4.9}
 Let $X$ be a u.l.f.\@ metric space and let $(a_n)_n$ be a sequence of compact operators in $\cst(X)$ so that $a_{\bar\lambda}\in \cst(X)$ for all $\bar\lambda\in \D^\N$. Then for all $\eps>0$ there is $r>0$ so that $a_{\bar\lambda}$ is $\eps$-$r$-approximable in $\cst(X)$ for all $\bar\lambda\in \D^\N$. 
 \end{lemma}
\begin{proof}
By Lemma~\ref{Lemma4.9BD}, we know that for all $\e>0$ there is $r$ such that $a_{\bar\lambda}$ is $\eps/2$-$r$-approximable in $\BD(X)$ for all $\bar\lambda\in \D^\N$. Given $\bar\lambda\in\D^\N$, Proposition~\ref{PropApprox} implies that $a_{\bar\lambda}$ is $\e$-$r$-approximable in $\cst(X)$. Since $\e$ and $\bar\lambda$ are arbitrary, we are done.
\end{proof}

 \begin{lemma}\label{LemmaCoarseLikeRA}
Let $X$ and $Y$ be u.l.f.\@ metric spaces and let $\Phi\colon\cst(X)\to \cst(Y)$ be a strongly continuous compact preserving linear map. If $\Phi\restriction \chi_F\cst(X)\chi_F$ is coarse-like for all finite $F\subset X$, then $\Phi$ is coarse-like.
\end{lemma}
\begin{proof}
Suppose that $\Phi\restriction\chi_E\cst(X)\chi_E$ is coarse-like for all finite $E\subseteq X$ but $\Phi$ is not coarse-like. Let $\e>0$ and $r>0$ be such that for every $s>0$ there is a contraction $a_s\in\cst(X)$ of propagation at most $ r$ such that $\Phi(a_s)$ is not $\e$-$s$-approximable.
\begin{claim}
For all cofinite $F\subset X$ and all $s>0$ there is a contraction $a\in \chi_F\cst(X)\chi_F$ with finite support and propagation at most $r$ so that $\Phi(a)$ is not $\eps/2$-$s$-approximable. 
\end{claim}
 
\begin{proof}
Fix $F$ and $s$. Let 
\[
E=\Big\{x\in X\mid d(x, X\setminus F)\leq r\Big\}.
\]
As $X\setminus F $ is finite and $X$ is u.l.f., $E$ is finite, and therefore $\Phi\restriction \chi_E\cst(X)\chi_E$ is coarse-like. Pick $s'>s$ so that $\Phi(\chi_Ea\chi_E)$ is $\eps/2$-$s'$-approximable for all contractions $a\in\cst(X)$. 

By the definition of $E$, we have that 
\[
\{(x,y)\in X^2\mid d(x,y)\leq r\}\subset (E\times E)\cup (F\times F).
\] 
Hence, if $a\in\cst(X)$ has propagation at most $ r$ then there is $b\in\chi_E\cst(X)\chi_E$ such that $a=b+\chi_Fa\chi_F$. If $a$ is a contraction, then $\norm{b}\leq 2$ and $\propg(b)\leq r$. Let $b\in\chi_E\cst(X)\chi_E$ be as above for $a=a_{s'}$. By our choice of $s'$, $\Phi(b)$ is $\eps/2$-$s'$-approximable. As $s'>s$, if $\Phi(\chi_Fa\chi_F)$ is $\e/2$-$s$-approximable, then $\Phi(a)$ is $\eps$-$s'$-approximable. This contradicts our choice of $a_s$, so $\Phi(\chi_Fa_s\chi_F)$ is not $\e/2$-$s$-approximable. By Proposition \ref{PropSOTConvApprox}, we can obtain a finite $G\subset X\setminus F$ so that $\Phi(\chi_Ga_s\chi_G)$ is not $\e/2$-$s$-approximable. This finishes that claim.
\end{proof} 

By the previous claim, there exists a disjoint sequence of finite subsets $(E_n)_n$ of $X$ and a sequence $(a_n)_n$ in $\cst(X)$ so that $a_n\in \cB(\ell_2(E_n,H))$, $\propg(a_n)\leq r$ and $\Phi(a_n)$ is not $\eps/2$-$n$-approximable for all $n\in\N$. Since the $E_n$'s are disjoint, for all $\bar\lambda\in\D^\N$, $a_{\bar\lambda}$ is a well defined element of $\cst(X)$. By Lemma \ref{Lemma4.9}, there is $s>0$ so that $\Phi(a_n)$ is $\eps/2$-$s$-approximable for all $n\in\N$, a contradiction.
\end{proof}

 If we substitute instances of $\cst(X)$ and of $\cst(Y)$ with $\BD(X)$ and $\BD(Y)$ in the proof of Lemma~\ref{LemmaCoarseLikeRA}, we obtain the following.
 \begin{lemma}\label{LemmaCoarseLikeRA2}
Let $X$ and $Y$ be u.l.f.\@ metric spaces and let $\Phi\colon\BD(X)\to \BD(Y)$ be a strongly continuous compact preserving linear map. If $\Phi\restriction \chi_F\BD(X)\chi_F$ is coarse-like for all finite $F\subset X$, then $\Phi$ is coarse-like.\qed
\end{lemma}
 
We are ready to prove our uniform approximability results.
 
 \begin{proof}[Proof of Theorem~\ref{ThmCoarseLikeBD}]
 Fix u.l.f metric spaces $X$ and $Y$. Recall that we need to show that any strongly continuous compact preserving $^*$-homomorphism $\Phi\colon\BD(X)\to \BD(Y)$ is coarse-like. By Lemma~\ref{LemmaCoarseLikeRA2}, it is enough to show that $\Phi\restriction\chi_F\BD(X)\chi_F$ is coarse-like for all such $\Phi$ and all finite $F\subseteq X$. As $\chi_{\{x\}}\BD(X)\chi_{\{y\}}\cong\cB(H)$ for all $x$ and $y$ in $X$, it is enough to show that any strongly continuous compact preserving $^*$-homomorphism $\Phi\colon \cB(H)\to \BD(Y)$ is coarse-like. Fix such $\Phi$. 
 
Let $(\xi_n)_n$ be an orthonormal basis for $H$. If $F\subseteq\N$, let $p_F$ be the projection onto $\overline{\spann}\{\xi_i\mid i\in F\}$. We write $p_n$ for $p_{\{1,\ldots,n\}}$. By Proposition \ref{PropSOTConvApproxBD}, it is enough to show that for all $\e>0$ there is $s>0$ so that for all $n\in\N$ and all contractions $a\in \cB(H)$, $\Phi(p_nap_n)$ is $\eps$-$s$-approximable in $\BD(Y)$. Suppose that this fails for $\eps\in (0,1)$. By compactness of the unit sphere of $p_n\cB(H)p_n$, we then have that 
\begin{itemize}
\item[$(\ast)$] For all finite $E\subset \N$ and all $s>0$ there is a contraction $a\in \cB(H)$ with finite support so that $p_Eap_E=0$ and $\Phi(a)$ is not $\eps/2$-$s$-approximable in $\BD(Y)$.
\end{itemize} 

\begin{claim}
For all cofinite $F\subset \N$ and all $s>0$ there is a contraction $a\in \cB(p_FH)$ with finite support so that $\Phi(a)$ is not $\eps/4$-$s$-approximable in $\BD(Y)$.
\end{claim}

\begin{proof}
Suppose the claim fails for a cofinite $F\subset\N$ and $s>0$. Let $A_0=\N\setminus F$. By $(\ast)$ there are an increasing sequence of finite subsets $(A_n)_n$ of $\N$ and a sequence of contractions $(a_n)_n$ in $\cB(H)$ so that $\supp(a_n)\subset A_n^2\setminus A_{n-1}^2$ and $\Phi(a_n)$ is not $\eps/2$-$n$-approximable in $\BD(Y)$ for all $n\in\N$. Since the claim fails, going to a subsequence and redefining $(a_n)_n$, we can assume that $\supp(a_n)\subset (A_n\setminus A_{n-1})\times A_0$ for all $n\in\N$ and that $\Phi(a_n)$ is not $\eps/8$-$n$-approximable in $\BD(Y)$ for all $n\in\N$ (otherwise, we could assume that $\supp(a_n)\subset A_0\times (A_n\setminus A_{n-1})$ and the proof would proceed similarly). Let $F_n=A_n\setminus A_{n-1}$. 

Let $(E_n)_n$ be a disjoint sequence of finite subsets of $\N$ so that $|E_n|=|A_0|$ for all $n\in\N$. For each $n\in\N$, let $b_n\in\cB(p_{A_0}H,p_{E_n}H)$ be a unitary. For each $n\in\N$, let $k(n)\in\N$ be so that $\Phi(b_n)$ can be $\eps/17$-$k(n)$-approximable in $\BD(Y)$ and let $m(n)\in\N$ be so that $\Phi(a_{m(n)})$ is not $\eps/8$-$(k(n)+n)$-approximable in $\BD(Y)$. Without loss of generality $(m(n))_n$ is strictly increasing.

Notice that $\supp(b_n a_{m(n)})\subset F_{m(n)}\times E_n$ for all $n\in\N$. As both $(E_n)_n$ and $(F_{m(n)})_n$ are disjoint sequences, we have that 
\[
\SOTh\sum_{n\in\N}\lambda_nb_na_{m(n)}\in \cB(H)
\]
for all $\lambda_n\in\D^\N$. As $\Phi$ is strongly continuous and compact preserving, each $\Phi(b_na_{m(n)})$ is compact and \[\SOTh\sum_{n\in\N}\lambda_n\Phi(b_na_{m(n)})\in \BD(Y)\]
for all $(\lambda_n)_n\in\D^\N$. By Lemma \ref{Lemma4.9BD}, there exists $s'>0$ so that $\Phi(b_na_{m(n)})$ is $\eps/16$-$s'$-approximable in $\BD(Y)$ for all $n\in\N$. Fix a sequence $(c_n)_n$ in $\BD(Y)$ so that $\propg(c_n)\leq s'$ and $\|\Phi(b_na_{m(n)})-c_n\|\leq \eps/16$ for all $n\in\N$.

As $b_n$ is unitary, each $\Phi(b_n^{-1})$ is $\eps/17$-$k(n)$-approximable in $\BD(Y)$. Fix $(d_n)_n$ in $\BD(Y)$ so that $\propg(d_n)\leq k(n)$ and $\|\Phi(b_n^{-1})-d_n\|\leq \eps/17$ for all $n\in\N$. Then, for all $n\in\N$,
\begin{align*}
\|\Phi(a_{m(n)})- & d_nc_n\|=&\\=&\norm{\Phi(b_n^{-1})\Phi(b_na_{m(n)})-d_nc_n}\\
\leq& \|\Phi(b_n^{-1})\Phi(b_na_{m(n)})-\Phi(b_n)^{-1} c_n\|+\|\Phi(b_n)^{-1} c_n-d_nc_n\|\\
\leq& \|\Phi(b_na_{m(n)})-c_n\|+\|\Phi(b_n)^{-1}-d_n\|\cdot\|c_n\|\\
\leq& \frac{\eps}{16}+\frac{\eps}{17}\Big(1+\frac{\eps}{16}\Big)\leq \frac{\eps}{8}.
\end{align*} 
As $\propg(d_nc_n)\leq k(n)+s'$, then $\Phi(a_{m(n)})$ is $\eps/2$-$(k(n)+s')$-approximable in $\BD(Y)$ for all $n\in\N$. For $n>s'$, this gives a contradiction.
\end{proof}

By the previous claim we can pick mutually disjoint finite sets $E_n\subseteq \N$ and a sequence of contractions $(a_n)_n$ so that $a_n\in \cB(P_{E_N}H)$ and $\Phi(a_n)$ is not $\eps/4$-$n$-approximable for all $n\in\N$. Since the $E_n$'s are mutually disjoint, $a_{\bar\lambda}\in\cB(H)$ for all $\bar\lambda\in\D^\N$. By Lemma \ref{Lemma4.9BD}, there exists $s>0$ so that $\Phi(a_n)$ is $\eps/4$-$s$-approximable in $\BD(Y)$ for all $n\in\N$, a contradiction.
 \end{proof}

\begin{proof}[Proof of Theorem~\ref{ThmCoarseLike}]
 Fix u.l.f metric spaces $X$ and $Y$, and an isomorphism $\Phi\colon\cst(X)\to \cst(Y)$. By Lemma~\ref{LemmaCoarseLikeRA}, it is enough to show that $\Phi\restriction\chi_F\cst(X)\chi_F$ is coarse-like for all finite $F\subseteq X$. Therefore, finiteness of $F$ implies that we only need to show that $\Phi\restriction\chi_{\{x\}}\cst(X)\chi_{\{y\}}$ is coarse-like for all $x$ and $y$ in $X$. To simplify the notation, assume $x=y$. 

We prove the following stronger statement:
\begin{enumerate}
\item[($*$)] For every $\e>0$ there is a finite $F\subseteq Y$ such that $\norm{\chi_F\Phi(a)\chi_F-\Phi(a)}<4\e$ for all positive contractions $a\in \chi_{\{x\}}\cst(X)\chi_{\{x\}}$. 
\end{enumerate} Notice that, since $\chi_F\Phi(a)\chi_F$ has propagation at most $\diam(F)$, for all $a\in\cst(X)$, ($*$) implies the desired result. We proceed by contradiction, so assume the statement in ($*$) fails for $\eps>0$.

Let $(\xi_n)_n$ be an orthonormal base for $H$, and let $p_n$ be the projection onto $\spann\{\xi_i\mid i\leq n\}$. For each $n\in\N$, let $q_n=\chi_{\{x\}}\otimes p_n$. %; so $(q_n)_n$ is an approximate identity of projections for $\chi_{\{x\}}\cst(X)\chi_{\{x\}}$. 
Given a finite $I\subset \N$, write $q_I= q_{\max I}-q_{\min I}$.

\begin{claim}
For every finite $F\subseteq Y$ and $n\in\N$ there is a positive contraction $a\in \chi_{\{x\}}\cst(X)\chi_{\{x\}}$ with $aq_n=q_na=0$ and 
\[
\norm{\chi_{Y\setminus F}\Phi(a)\chi_{Y\setminus F}}>\e^2.
\]
\end{claim}
\begin{proof}
Fix a finite $F\subset Y$ and $n\in\N$. Since $q_n \cst(X) q_n$ is finite dimensional and each element of $q_n\cst(X)q_n$ has finite rank, there is a finite $G\subseteq Y$ such that whenever $a\in q_n\cst(X)q_n$ is a contraction and $G'\supset G$, then 
\[\norm{\Phi(a)-\chi_{G'}\Phi(a)}<\e^2.\]
 Fix $G'=G\cup F$. By our choice of $\eps$, there is a positive contraction $b\in \chi_{\{x\}}\cst(X)\chi_{\{x\}}$ such that $\norm{\chi_{G'}\Phi(b)\chi_{G'}-\Phi(b)}>4\e$. In particular, the triangular inequality implies that \[\norm{\chi_{Y\setminus G'}\Phi(b)}=\norm{\chi_{G'}\Phi(b)-\Phi(b)}>2\e.\] 

Let $a=(\chi_{\{x\}}-q_n)b^2(\chi_{\{x\}}-q_n)$. So $a$ is a positive contraction. Assume for a contradiction that $\norm{\chi_{Y\setminus F}\Phi(a)\chi_{Y\setminus F}}\leq \e^2$. Then $\norm{\chi_{Y\setminus G'}\Phi(a)\chi_{Y\setminus G'}}\leq \e^2$. Since 
\[
b^2=a+q_nb^2q_n+q_nb^2(\chi_{\{x\}}-q_n)+(\chi_{\{x\}}-q_n)b^2q_n,
\]
 we have that $\norm{\chi_{Y\setminus G'}\Phi(b^2)\chi_{Y\setminus G'}}\leq 4\e^2$, so $\norm{\chi_{Y\setminus G'}\Phi(b)}\leq 2\e$. This is a contradiction.
\end{proof}

Notice that $a=\SOTh\lim_m(q_m-q_n)a(q_m-q_n)$, for all $a\in \chi_{\{x\}}\cst(X)\chi_{\{x\}}$ with $q_na=aq_n=0$. Therefore, the previous claim can be used to produce a sequence $(a_n)_n$ of contractions in $\chi_{\{x\}}\cst(X)\chi_{\{x\}}$, a sequence of natural numbers $(k(n))_n$, and sequences $(F_n)_n$ and $(I_n)_n$ of finite subsets of $Y$ and $\N$, respectively, so that $(F_n)_n$ is a disjoint sequence, $\max I_n<\max I_ {n+1}$ for all $n\in\N$, and 
\[
\norm{(\chi_{F_n}\otimes p_{k(n)})\Phi(q_{I_n}a_na^*_nq_{I_n})(\chi_{F_n}\otimes p_{k(n)})}>\e^2/4
\]
for all $n\in\N$. Hence, the \cstar-equality gives that 
\[
\norm{\chi_{F_n}\otimes p_{k(n)}\Phi(q_{I_n}a_n)}>\e/2
\]
for all $n\in\N$. 

As both $(\chi_{F_n}\otimes p_{k(n)})_n$ and $(\Phi(q_{I_n}a_n))_n$ are sequences of compact operators converging to zero in the strong operator topology, passing to a subsequence if necessary, we can assume that
\[
\norm{\chi_{F_n}\otimes p_{k(n)}\Phi(q_{I_m}a_m)}\leq 2^{-n-3}\eps
\]
for all $n\neq m$ in $\N$. As $(F_n)_n$ is a disjoint sequence, $b=\SOTh\sum_n\chi_{F_n}\otimes p_{k(n)}$ exists and it clearly belongs to $\cst(Y)$. Let $c=\Phi^{-1}(b)$; in particular $c\in \cst(X)$. Then 
\begin{align*}
\|cq_{I_m}\|&\geq \|cq_{I_m}a_m\|=\|b\Phi(q_{I_m}a_m)\|\\
&\geq \norm{\chi_{F_n}\otimes p_{k(n)}\Phi(q_{I_m}a_m)}-\sum_{n\neq m}\norm{\chi_{F_n}\otimes p_{k(n)}\Phi(q_{I_m}a_m)} \\
&\geq \eps/4
\end{align*}
for all $m\in\N$. As $(I_n)_n$ are disjoint, this contradicts the fact that $c\in \cst(X)$, i.e., that $c$ is locally compact. 
\end{proof}
 \begin{remark}
 We used positivity and the fact that the $q_n$'s are projection in the above proof. We would not need it, by playing with functional analysis. So if we ever want, we could follow the strategy highlighted above to show that if $\Phi\colon\cst(X)\to\cst(Y)$ is a strongly continuous linear compact preserving map with the property that for every sequence of operators $(a_n)$ we have that if $\Phi(a_n)$ converges strongly to $b$ in $\cst(Y)$ then there is $c\in\cst(X)$ such that $a_n$ converges strongly to $c$. This shows that having a strongly continuous sequence which is converging outside $\cst(X)$ and which is sent to a strongly converging sequence converging inside $\cst(Y)$ is indeed the only obstruction to generalising Theorem~\ref{ThmCoarseLikeURA} to the Roe algebra setting (see, e.g., the discussion after Proposition~\ref{PropNotCoarseLike}).
 \end{remark}

We finish the section introducing a weaker version of coarse-likeness for which Theorem \ref{ThmCoarseLikeURA} has an equivalent for Roe algebras.
 
\begin{definition}\label{DefiAsympCoarseLikeDEF}
Let $X$ and $Y$ be metric spaces. A map $\Phi\colon\cst(X)\to \cst(Y)$ is \emph{asymptotically coarse-like} if for all $\eps>0$ and $r>0$ there are $s>0$ and a cofinite $X'\subset X$ so that $\Phi(a)$ is $\eps$-$s$-approximable in $\cst(Y)$ for all contractions in $a\in \cst(X')$ with $\propg(a)\leq r$. 
\end{definition}

 \begin{theorem}\label{ThmAsympCoarseLike}
Let $X$ and $Y$ be u.l.f.\@ metric spaces. Every strongly continuous compact preserving linear map $\Phi\colon\cst(X)\to \cst(Y)$ is asymptotically coarse-like. 
\end{theorem} 

 \begin{proof} 
Suppose this fails. So there is $\eps>0$ and $r>0$ so that for all $n\in\N$ and all cofinite $X'\subset X$, there is a contraction $a\in \cst(X')$ with $\propg(a)\leq r$ so that $\Phi(a)$ is not $\eps$-$n$-approximable. Then, by Proposition \ref{PropSOTConvApprox}, we can pick a disjoint sequence $(X_n)_n$ of finite subsets of $X$ and a sequence $(a_n)_n$ of contractions so that $a_n\in \cst(X_n)$, $\propg(a_n)\leq r$, and $\Phi(a_n)$ is not $\eps$-$n$-approximable for all $n\in\N$. 

Since each $X_n$ is finite and $\Phi$ preserves the compacts, each $\Phi(a_n)$ is compact. Hence, as $(X_n)_n$ is a disjoint sequence and as each $a_n$ has propagation at most $r$, strong continuity of $\Phi$ gives that $\SOTh\sum_n\lambda_n\Phi(a_n)\in \cstu(Y)$ for all $(\lambda_n)_n\in \D^\N$. Therefore, Lemma \ref{Lemma4.9} implies that there is $s>0$ to that each $\Phi(a_n)$ is $\delta$-$s$-approximable; contradiction.
\end{proof}
 
\begin{remark}
Although we will prove Theorems \ref{thm:main} and \ref{thm:WW} below using Theorem \ref{ThmCoarseLike}, we point out that both those results could be obtained (in a very similar way) using Theorem \ref{ThmAsympCoarseLike} above instead of Theorem \ref{ThmCoarseLike}. 
\end{remark}

\section{The multiplier algebra of $\cst(X)$}\label{SecMult}

In this short section, we characterize $\BD(X)$ as the multiplier algebra of $\cst(X)$. As a consequence, this shows that $\Inn(\BD(X))$ is a normal subgroup of $\Aut(\cst(X))$. Since all automorphisms of $\cst(X)$ are strongly continuous, being induced by a unitary in $\cB(\ell_2(X,H))$ (see e.g., \cite[Lemma 3.1]{SpakulaWillett2013AdvMath}), we have that all automorphisms of $\cst(X)$ extend to automorphisms of $\BD(X)$. An operator algebraist used to work with multipliers would not be surprised by this result, and may even find it obvious. However, we do not know of a proof that does not use uniform approximability in some way.
 
\begin{theorem}\label{prop:mult}
Let $X$ be a u.l.f.\@ metric space. Then $\BD(X)=\mathcal M(\cst(X))$.
\end{theorem}
\begin{proof}
We use the characterisation of the multiplier algebra given in \cite[II.7.3.5]{Blackadar.OA}. As $\cst(X)$ is already represented faithfully on $\ell_2(X,H)$, the multiplier algebra of $\cst(X)$ coincides with its idealizer, that is, 
\[
\mathcal M(\cst(X))=\Big\{b\in\cB(\ell_2(X,H))\mid \forall a\in\cst(X),\ ba, ab\in\cst(X)\Big\}.
\]
As $\cst(X)$ is an ideal in $\BD(X)$, we clearly have that $\BD(X)\subseteq\mathcal M(\cst(X))$.
\begin{claim}\label{claim:multipliers}
Let $b\in\mathcal M(\cst(X))$, $\eps>0$, and let $F\subseteq X$ be finite. Then there is a finite $G\subset X$ such that $\norm{\chi_Gb\chi_F-b\chi_F}<\e$ and $\norm{\chi_Fb\chi_G-\chi_Fb}<\e$. In particular, $b\chi_F $ and $\chi_Fb$ belong to $ \BD(X)$.
\end{claim}
\begin{proof}
Suppose not. Then, without loss of generality, we assume that there is a sequence $(G_n)_n$ of disjoint finite subsets $X$ so that $\norm{\chi_{G_n}b\chi_F}>\e/2$ for all $n\in\NN$. For each $n\in\N$, fix a unit vector $\xi_n$ such that $\norm{\chi_{G_n}b\chi_F\xi_n }>\e/2$ and let $p_n$ be a finite rank projection in $\cB(H)$ so that $\|(\chi_{G_n}\otimes p_n)b\chi_F\xi_n\|>\eps/2 $. Then $\SOTh\sum_n(\chi_{G_n}\otimes p_n)\in\cst(X)$, so, as $b\chi_F\in\mathcal M(\cst(X))$, then $\SOTh\sum_n(\chi_{G_n}\otimes p_n)b\chi_F\in\cst(X)$. Fix $k\in \N$ such that $\SOTh\sum_n(\chi_{G_n}\otimes p_n)b\chi_F$ can be $\e/2$-$k$-approximated in $\cst(X)$ and $m$ large enough so that $d(G_m,F)>k$. Since $\chi_{G_m}\otimes p_m$ and $\chi_F$ have propagation $0$, $(\chi_{G_{m}}\otimes p_m)b\chi_F$ can be $\e/2$-$k$-approximated in $\cst(X)$. Let $c\in \cst(X)$ be an element with propagation at most $k$ so that $\norm{c-(\chi_{G_{m}}\otimes p_m)b\chi_F}<\e/2$. Since $c$ has propagation at most $k$ and $d(G_{m},F)>k$, then $(\chi_{G_{m}}\otimes p_m)c\chi_F=0$. Hence $\norm{(\chi_{G_{m}}\otimes p_m)b\chi_F}<	\e/2$, a contradiction. 
\end{proof}

In order to get a contradiction, suppose $b\in\mathcal M(\cst(X))$ is such that there is $\e>0$ for which $b$ cannot be $\e$-$n$-approximated for every $n\in\N$. We will construct two sequences $(F_n)_n$ and $(p_n)_n$ of finite subsets of $X$ and finite-rank projections in $\mathcal B(H)$, respectively, with the following properties:
\begin{itemize}
\item the $F_n$ are disjoint,
\item for all $n$, 
\[
\norm{(\chi_{\bigcup_{m\neq n}F_m}\otimes 1_H)b(\chi_{F_n}\otimes p_n)}<2^{-n},
\]
\item each $(\chi_{F_n}\otimes p_n)b(\chi_{F_n}\otimes p_n)$ cannot be $\e/2$-$n$-approximated.
\end{itemize}
We do this by induction. Let $(q_n)_n$ be a sequence of finite-rank projections in $\cB(H)$ which is converging strongly to $1_H$. Since $b$ cannot be $\e$-$0$-approximated and 
\[
b=\SOTh\lim_{\substack{F\subseteq X\text{ finite }\\n\in\NN}}(\chi_F\otimes q_n)b(\chi_F\otimes q_n),
\]
 Proposition~\ref{PropSOTConvApproxBD} gives a finite $F_0\subseteq X$ and $n\in\NN$ such that $(\chi_F\otimes q_n)b(\chi_F\otimes q_n)$ cannot be $\e$-$0$-approximated. Let $p_0=q_n$. 
 
 We now make the inductive step: suppose that $p_0,\ldots,p_n$ and $F_0,\ldots,F_n$ have been defined. Let $b_n=b\chi_{\bigcup_{m\leq n}F_n}$. Using the previous claim, pick a finite $G\subseteq X$ such that $\norm{\chi_Gb_n-b_n}<2^{-n-1}$. By Claim \ref{claim:multipliers}, $\chi_{X\setminus G}b\chi_G+\chi_Gb\chi_{X\setminus G}+\chi_Gb\chi_G\in\BD(X)$, so there is $n'>n$ such that $\chi_{X\setminus G}b\chi_G+\chi_Gb\chi_{X\setminus G}+\chi_Gb\chi_G$ can be $\e/2$-$n'$-approximated. As 
\[
b=\chi_{X\setminus G}b\chi_{X\setminus G}+\chi_{X\setminus G}b\chi_G+\chi_Gb\chi_{X\setminus G}+\chi_Gb\chi_G
\] 
 cannot be $\e$-$n'$-approximated, $\chi_{X\setminus G}b\chi_{X\setminus G}$ cannot be $\e/2$-$n'$-approximated. As
\[
\chi_{X\setminus G}=\SOTh\lim_{\substack{F\subseteq X\setminus G \text{ finite}\\ i\in\NN}}\chi_{F}\otimes q_i,
\]
 Proposition~\ref{PropSOTConvApproxBD} gives a finite $F_{n+1}$ and $i\in\NN$ such that $(\chi_{F_{n+1}}\otimes q_i)b(\chi_{F_{n+1}}\otimes q_i)$ cannot be $\e/2$-$n'$-approximated. Setting $p_{n+1}=q_i$ concludes the construction.

Let now $c=\SOTh\sum_n(\chi_{F_n}\otimes p_n)$ and notice that
\[
cbc=\SOTh\sum_n(\chi_{F_n}\otimes p_n)b(\chi_{F_n}\otimes p_n)+d
\]
where 
\[
d=\SOTh\sum_{ n\in \N } \Big((\chi_{F_n}\otimes p_n)b\sum_{m\neq n} \chi_{F_m}\otimes p_m\Big).
\]
By our choice of $(F_n)_n$ and $(p_n)_n$, we have that 
\begin{eqnarray*}
\norm{\Big(1_X\otimes 1_H-\sum_{m\leq n}\chi_{F_m}\otimes p_m\Big)d}&\leq& \sum_{m>n}\norm{(\chi_{F_m}\otimes p_m)b\sum_{n'\neq m}\chi_{F_{n'}}\otimes p_{n'}}\\&\leq&\sum_{m>n}2^{-m}\leq 2^{-n+1}
\end{eqnarray*}
for all $n\in\N$. Hence $d$ is compact, so $d\in \cst(X)$. Let 
\[
b'=\SOTh\sum_n(\chi_{F_n}\otimes p_n)b(\chi_{F_n}\otimes p_n).
\]
As $b\in\mathcal M(\cst(X))$, it follows that $cbc\in\cst(X)$. Hence, as $d\in\cst(X)$, we have that $b'\in\cst(X)$. Pick $n$ such that $b'$ can be $\e/2$-$n$-approximated. Since $\chi_{F_n}\otimes p_n$ has propagation $0$, $(\chi_{F_n}\otimes p_n)b'(\chi_{F_n}\otimes p_n)$ can be $\e/2$-$n$-approximated. This is a contradiction since $(\chi_{F_n}\otimes p_n)b'(\chi_{F_n}\otimes p_n)=(\chi_{F_n}\otimes p_n)b(\chi_{F_n}\otimes p_n)$.
\end{proof} 
%\iffalse
%A: Do we make anything of this?

The following is a simple consequence of Theorem \ref{prop:mult}.

\begin{corollary}\label{CorNormalSubgroup}
Let $X$ be a u.l.f.\@ metric space. Any $\Phi\in\Aut( \cst(X))$ extends to an automorphism of $\BD(X)$. Moreover, $\Inn(\BD(X))$ is a normal subgroup of $\Aut(\cst(X))$.\qed
\end{corollary}

 \begin{remark}
If one is only interested in Corollary \ref{CorNormalSubgroup}, Theorem \ref{prop:mult} is not necessary. In fact, Theorem \ref{ThmCoarseLike} gives us an easy proof of Corollary \ref{CorNormalSubgroup}. We outline it here as an example of the power of Theorem \ref{ThmCoarseLike}.

Fix $\Phi\in \Aut(\cst(X))$ and let $u\in \cB(\ell_2(X,H))$ be a unitary so that $\Phi=\Ad(u)$ (e.g., \cite[Lemma 3.1]{SpakulaWillett2013AdvMath}). Fix $r>0$ and $\eps>0$ and let $a\in \BD(X)$ be a contraction with $\propg(a)\leq r$. Then, $a$ can be easily written as $a=\SOTh\lim_n a_n$ where $(a_n)_n$ is a sequence of contractions in $\cst(X)$ with $\propg(a_n)\leq r$ for all $n\in\N$. By Theorem \ref{ThmCoarseLike}, there is $s=s(r,\eps)>0$ and a sequence $(b_n)_n$ in $\cst(X)$ so that $\propg(b_n)\leq s$ and $\|\Phi(a_n)-b_n\|\leq \eps$ for all $n\in\N$. Using weak operator compactness and going to a subsequence, we can assume that $b=\WOTh\lim_nb_n$ exists. Clearly, $\propg(b)\leq r$ and $\|\Ad(u)(a)-b\|\leq \eps$. As $\eps$ was arbitrary, this shows that $\Ad(u)(a)\in \BD(X)$. We leave the remaining details to the reader. 
\end{remark}

\section{Proof of the main result}\label{SectionMakingIsoCoarseLike}

We use the uniform approximability results of \S\ref{SecUnifApprox} to prove Theorems~\ref{thm:main} and \ref{ThmCartan}.

\subsection{Technical lemmas}\label{SubsectionTechnical}
We prove several technical lemmas in this subsection. Their proofs are inspired by techniques in \cite[Section 6]{WhiteWillett2017}.

\begin{definition}
A u.l.f.\@ metric space $X$ has the \emph{operator norm localisation property} (\emph{ONL}) if for all $s>0$ and all $\rho\in (0,1)$ there is $r>0$ so that if $a\in \cB(\ell_2(X,H))$ has propagation at most $s$ then there exists a unit vector $\xi\in \ell_2(X,H)$ with $\diam(\supp(\xi))\leq r$ so that $\|a\xi\|\geq \rho\|a\|$.\footnote{Recall, $\supp(\xi)=\{x\in X\mid \xi(x)\neq 0\}$ for a given $\xi:X\to H$.} 
\end{definition}

By \cite[Theorem 4.1]{Sako2014}, a u.l.f.\@ metric space has property A if and only if it has ONL. The following assumption will be recurrent:

\begin{assumption}\label{AssumptionIsoONL} Let $X$ and $Y$ be u.l.f.\@ metric spaces with ONL, $H$ be a separable infinite-dimensional Hilbert spaces and let $\Phi\colon\cst(X)\to \cst(Y)$ be an isomorphism. For $\delta>0$, and projections $p\in \cst(X)$ and $q\in\cst(Y)$ of rank 1, we let \[Y_{p,\delta}=\{y\in Y \mid \|\Phi(p)\chi_{\{y\}}\|\geq \delta\}\]
and 
\[X_{q,\delta}=\{x\in X\mid \|\Phi^{-1}(q)\chi_{\{x\}}\|\geq \delta\}.\]
\end{assumption}
 
We point out that isomorphisms between Roe algebras are automatically strongly continuous and rank preserving (see \cite[Lemma 3.1]{SpakulaWillett2013AdvMath}). This will be used with no further mention in the proofs of the lemmas below.

\begin{lemma}\label{Lemma1}
In the setting of Assumption \ref{AssumptionIsoONL}, for all $r>0$ and $\eps>0$ there is $t>0$ so that for all projections $p\in \cstu(X)$ and $q\in\cst(Y)$ with propagation at most $r$ there is $E\subset X$ with $\diam(E)\leq t$ so that 
\[\|\Phi(p)q\chi_E\|\geq (1-\eps)\|\Phi(p)q\|-\eps.\]
\end{lemma}

\begin{proof}
Fix $\eps>0$. Theorem \ref{ThmCoarseLike} gives $s>0$ so that $\Phi(p)$ is $\eps/2$-$s$-approximable for all projections $p\in \cst(X )$ with $\propg(p)\leq r$. Fix projections $p\in \cst(X)$ and $q\in \cst(Y)$ with propagation at most $r$. So there is $b\in \cst(Y)$ with $\propg(b)\leq s+r$ so that $\|\Phi(p)q-b\|\leq \eps/2$. As $Y$ has ONL, there exists $t>0$ (depending only on $\eps$ and $r$) and a unit vector $\xi\in \ell_2(Y,H)$ so that $\supp(\xi)\leq t$ and $\|b\xi\|\geq (1-\eps)\|b\|$. Let $E=\supp(\xi)$. Hence, 
\begin{align*}
\|\Phi(p)q\chi_E\|&\geq\|b\chi_E\|- \|\Phi(p)q\chi_E-b\chi_E\|\\
&\geq (1-\eps)\|b\| -\frac{\eps}{2}\\
&\geq (1-\eps)\|\Phi(p)q\|-(1-\eps)\|\Phi(p)q-b\|-\frac{\eps}{2}\\
&\geq (1-\eps)\|\Phi(p)q\|-\eps,
\end{align*}
and we are done.
\end{proof}

Given $n\in\N$, $r\geq 0$, we write 
 \[\Proj_{n,r}(X)=\Big\{p\in \Proj(\cst(X))\mid \rank(p)\leq n\text{ and }\propg(p)\leq r\Big\}.\]
We define $\Proj_{n,r}(Y)$ analogously.

\begin{lemma}\label{LemmaYxndeltaBoundedDiam}
In the setting of Assumption \ref{AssumptionIsoONL}, for all $r>0$ and $\delta>0$, we have that \[\sup_{p\in \Proj_{1,r}(X) }\diam(Y_{p,\delta})<\infty.\]
\end{lemma}

\begin{proof}
By Lemma \ref{Lemma1}, there is $t>0$ so that for all $p\in \Proj_{1,r}(X)$ there is $E\subset Y$ with $\diam(E)\leq t$ so that $\|\Phi(p)\chi_E\|^2> 1-\delta^2$. Fix $p\in \Proj_{1,r}(X)$ and let $E$ be as above. As $\Phi$ is rank preserving, $\Phi(p)$ has rank 1. Pick a unit vector $\xi\in \ell_2(Y,H)$ in the range of $\Phi(p)$. Then \[\|\Phi(p)\chi_F\|=\|\chi_F\Phi(p)\|=\|\chi_F\xi\|\] for all $F\subset Y$. In particular, $\|\chi_E\xi\|^2> 1-\delta^2$. If $y\not\in E$, then 
\[\|\Phi(p)\chi_{\{y\}}\|^2=\|\chi_{\{y\}}\xi\|^2\leq \|\xi\|^2-\|\chi_{Y\setminus \{y\}}\|^2\leq 1-\|\chi_E\xi\|^2 < \delta^2.\]
So, $y\not\in Y_{p,\delta}$. This shows that $Y_{p,\delta}\subset E$ and we must have $\diam(Y_{p,\delta})\leq t$.
\end{proof}

\begin{lemma}\label{Lemma2}
In the setting of Assumption \ref{AssumptionIsoONL}, for all $r>0$ and $\eps>0$ there is $\delta>0$ so that 
 \[\inf_{p\in \Proj_{1,r}(X)}\|\Phi(p)\chi_{Y_{p,\delta}}\|\geq 1-\eps.\]
\end{lemma}

\begin{proof}
Fix $r>0$ and $\eps\in (0,1)$. By Lemma \ref{Lemma1}, there is $t>0$ so that for all $p\in \Proj_{1,r}(X )$ there is $E\subset Y$ with $\diam(E)\leq t$ such that $\|\Phi(p)\chi_E\|^2> 1-\eps$. As $Y$ is u.l.f., $N=\sup_{y\in Y}|B_t(y)|$ is finite. Let $\delta\in (0, \sqrt{(\eps-\eps^2)/N})$ and fix $p\in \Proj_{1,r}(X)$. Let $\xi\in \ell_2(Y,H)$ be a unit vector in the range of the rank 1 projection $\Phi(p)$. Then, picking $E$ as above, we have that
\begin{align*}
\|\Phi(p)\chi_{ Y_{p,\delta}}\|^2& =\|\chi_{Y_{p,\delta}}\xi\|^2\\
&\geq \|\chi_{E\cap Y_{p,\delta}}\xi\|^2\\
&\geq \|\chi_E\xi\|^2-\|\chi_{ E\setminus Y_{p,\delta}}\xi\|^2\\
&= \|\Phi(p)\chi_E\|^2-\|\Phi(p)\chi_{ E\setminus Y_{p,\delta}}\|^2\\
&\geq 1-\eps-\delta^2N\\
&\geq (1-\eps)^2,
\end{align*}
and we are done.
\end{proof}

Before stating the next lemma, we need to introduce some technical notation. Given positive reals $t$, $r$ and $k$, we denote by $\cD_{t,r,k}(X)$ the set of all families $(p_i)_{i\in \N}$ of orthogonal projections in $\cst(X)$ satisfying:
\begin{enumerate}
\item each $p_i$ has rank at most 1,
\item each $p_i$ has propagation at most $r$, and
\item $|\{i\in\N\mid p_i\chi_E\neq 0\}|\leq k$ for any $E\subset X$ with $\diam(E)\leq t$.
\end{enumerate}
If $(p_i)_{i\in \N}\in \cD_{t,r,k}(X)$, then $\SOTh\sum_{i\in \N}p_i \in \cst(X)$. Let $\bar p=\SOTh\sum_{i\in\N}p_i$. By abuse of notation, we identify $\bar p$ with $(p_i)_{i\in \N}$ and write $\bar p\in \cD_{t,r,k}(X)$. We define $\cD_{t,r,k}(Y)$ analogously, and write $\bar q\in\cD_{t,r,k}(Y)$ if $\bar q=\SOTh\sum_{i\in\N}q_i$ where $(q_i)_{i\in\N}$ is in $\cD_{t,r,k}(Y)$.

\begin{remark}\label{RemarkDktX}
A word on the prototypical elements of $\cD_{t,r,k}(X)$ is 	useful: If $(P_x)_{x\in X}$ is a family of projections on $H$ and $p_x=\chi_{\{x\}}\otimes P_x$, then $(p_x)_{x\in X}$ belongs to $\cD_{t,r,k}(X)$ for any $t$ and $r\geq 0$, where $k=\sup_{x\in X}|B_t(x)|$. More generally, let $(X_n)_n$ be a sequence of disjoint subsets of $X$ with $r=\sup_n\diam(X_n)<\infty$, $\ell\in\N$, and for each $n\in\N$ let $(p_{n,i})_{i=1}^\ell$ be a family of orthogonal projections in $\cB(\ell_2(X_n,H))$ of rank at most 1. Then, for any $t>0$, there is $k>0$ so that $\bar p=((p_{n,i})_{i=1}^\ell)_{n\in\N}\in \cD_{t,r,k}(X)$. Indeed, first notice that the propagation of each $p_{n,i}$ is at most $r$. Moreover, since $X$ is u.l.f., there is $k_0\in\N$ so that any $E\subset X$ with $\diam(E)\leq t$ intersects at most $k_0$-many $X_n$'s. Therefore, $p_{n,i}\chi_E\neq 0$ for at most $k_0\ell$-many $(n,i)$'s. So $\bar p\in \cD_{t,r,k}(X)$ for $k=k_0\ell$. Notice that $k$ depends only on $\ell$ and $t$ (i.e., it depends on neither $\bar p$ not $r$).
\end{remark}

Let $t,r,k$ and $\delta$ be positive reals, $\bar p\in \cD_{t,r,k}(X)$ and $\bar q=(q_i)_{i\in\N}\in \cD_{t,r,k}(Y)$, we write 
\[Y_{\bar p,\delta}=\bigcup_{i\in\N} Y_{p_i,\delta}\text{ and }X_{\bar q,\delta}=\bigcup_{i\in\N} X_{q_i,\delta}.\]

\begin{lemma}\label{Lemma3}
In the setting of Assumption \ref{AssumptionIsoONL}, for all $r>0$ and $\eps>0$, there is $t>0$ so that for all $k\in\N$, there is $\delta>0$ so that 
\[ \sup_{\bar p\in \cD_{t,r,k}(X )} \|\Phi (\bar p)\chi_{Y\setminus Y_{\bar p,\delta}}\|\leq \eps.\]
\end{lemma}

\begin{proof}
Let $\theta=\eps/(2+\eps)$. By Lemma \ref{Lemma1} applied to $\Phi^{-1}$, $r$ and $\theta$, there exists $t>0$ so that for all projections $p\in \cst(X)$ and $q\in \cst(Y)$ with propagation at most $r$ there is $E\subset X$ with $\diam (E)\leq t$ so that 
\[\|\Phi(p\chi_E)q\|=\|\Phi^{-1}(q)p\chi_E\|\geq (1-\theta)\|\Phi^{-1}(q)p\|-\theta=(1-\theta)\|\Phi(p)q\|-\theta.\]
Fix $k\in\N$. By Lemma \ref{Lemma2}, pick $\delta>0$ so that $\|\Phi(p)\chi_{Y_{p,\delta}}\|^2\geq 1-(\theta/k)^2$ for all $p\in \Proj_{1,r}(X)$. 

Fix $\bar p\in\cD_{t,r,k}(X)$. For each $i\in\N$ with $p_i\neq 0$, $p_i$ has rank 1; hence so has $\Phi(p_i)$. For each $i\in\N$, pick a unit vector $\xi_i\in \ell_2(Y,H)$ in the range of $\Phi( p_i)$. Then 
\begin{align*}
\|\Phi( p_i)(1-\chi_{Y_{ p_i,\delta}})\|^2&=\|(1-\chi_{Y_{ p_i,\delta}})\xi_i\|^2=\|\xi_i\|^2-\|\chi_{Y_{p_i,\delta}}\xi_{i}\|^2\\
&= 1-\|\Phi(p_i)\chi_{Y_{p_i,\delta}}\|^2\leq(\theta/k)^2. 
\end{align*}
In particular, for all $i\in\N$
\[
\|\Phi( p_i)(1-\chi_{Y_{ \bar p,\delta}})\|\leq \|\Phi( p_i)(1-\chi_{Y_{ p_i,\delta}})\|\leq \theta/k.
\] 
 Let $C= Y\setminus Y_{\bar p,\delta}$ and $Q$ be a finite rank projection on $H$. So $q=\chi_C\otimes Q$ is a projection in $\cst(Y)$ and $\propg(q)=0$. As $\propg(\bar p)\leq r$, our choice of $t$ gives $E\subset Y$ with $\diam(E)\leq t$ so that
 \[\|\Phi(\bar p)q\|\leq \frac{\|\Phi(\bar p\chi_E)q\|+\theta}{1-\theta}.\]
 Therefore, as $\bar p\in\cD_{t,r,k}(X)$, we must have 
 \begin{align*}
 \|\Phi( \bar p)q\|&\leq \frac{k \cdot \sup_{i\in\N}\|\Phi( p_i)q\|+\theta}{1-\theta}\\
 &\leq \frac{k\cdot \sup_{i\in\N}\|\Phi( p_i)(1-\chi_{Y_{ \bar p,\delta}})\|+\theta}{1-\theta}\\
 &\leq \frac{2\theta}{1-\theta}
 \end{align*}
 As $Q$ was an arbitrary finite rank projection on $H$, this shows that 
\[
\|\Phi( \bar p)(1-\chi_{Y_{ \bar p,\delta}})\| \leq \frac{2\theta}{1-\theta}\leq \eps,\]
and we are done.
\end{proof}

We need one more lemma before presenting the proof of Theorem \ref{ThmCartan}.

\begin{lemma}\label{LemmaDistYpdelta}
In the setting of Assumption \ref{AssumptionIsoONL}, for every positive reals $\delta$ and $t$ there exists $s>0$ so that for all $E\subset X $ with $\diam(E)\leq t$ and all rank 1 projections $p$ and $q$ in $\cB(\ell_2(E,H))$ we have that $\partial(Y_{p,\delta},Y_{q,\delta})\leq s$. 
\end{lemma}

\begin{proof}
Suppose the lemma fails for $\delta$ and $t$. Without loss of generality, assume $\delta\in (0,1)$. Then there are a sequence of disjoint subsets $(E_n)_n$ of $X$, and sequences of rank 1 projections $(p_n)_n$ and $(q_n)_n$ so that 
\begin{enumerate}
\item $\diam(E_n)\leq t$ for all $n\in\N$,
\item $p_n,q_n\in \cB(\ell_2(E_n,H))$ for all $n\in\N$, and
\item $\partial(Y_{p_n,\delta},Y_{q_n,\delta})>n$ for all $n\in\N$. 
\end{enumerate}
For each $n\in\N$, let $a_n\in \cB(\ell_2(E_n,H))$ be a partial isometry so that $a_na_n^*=p_n$ and $a_n^*a_n=q_n$.

Let $\gamma>0$ be so that $\delta=\gamma(2-\gamma)$. As $\Phi$ is coarse-like (Theorem \ref{ThmCoarseLike}), there is $s>0$ so that $\Phi(a)$ is $\gamma/4$-$s$-approximable for all contractions $a\in \cst(X)$ with $\propg(a)\leq t$. Notice that all operators in each $\cB(\ell_2(E_n,H))$ must have propagation at most $t$. Hence, for each $n\in\N$ pick $b_n\in \cst(Y)$ so that $\propg(b_n)\leq s$ and $\|\Phi(a_n)-b_n\|\leq \gamma/4$. Since $Y$ has ONL, there are $s'>0$ and a sequence of subsets $(A_n)_n$ of $Y$ so that $\diam(A_n)\leq s'$ and $\|b_n\chi_{A_n}\|\geq 1-\gamma/2$ for all $n\in\N$. As $\propg(b_n\chi_{A_n})\leq s$ for all $n\in\N$, we can use that $Y$ has ONL once again in order to obtain a sequence of subsets $(B_n)_n$ of $Y$ so that $\propg(B_n)\leq s'$ and $\|\chi_{B_n}b_n\chi_{A_n}\|\geq 1-3\gamma/4$ for all $n\in\N$. Hence, for all $n\in\N$, 
\begin{align}\label{EqAnBn}
\|\chi_{B_n}\Phi(a_n)\chi_{A_n}\|\geq \|\chi_{B_n}b_n\chi_{A_n}\|-\|\Phi(a_n)-b_n\|> 1-\gamma.
\end{align}
Hence, as $\Phi(a_n)$ is $\gamma/4$-$s$-approximable for all $n\in\N$, this implies that $d(A_n,B_n)\leq s$ for all $n\in\N$. Therefore, \[\diam(A_n\cup B_n)\leq s+2s'\] for all $n\in\N$. As each $a_n$ is a partial isometry, it follows that 
\begin{align*}
\|\Phi(q_n)\chi_{A_n}\|\geq \|\Phi(a_nq_n)\chi_{A_n}\|=\|\Phi(a_na_n^*a_n)\chi_{A_n}\|
=\|\Phi(a_n)\chi_{A_n}\|> 1-\gamma
\end{align*}
for all $n\in\N$. Similarly, we have that $\|\chi_{B_n}\Phi(p_n)\|\geq 1-\gamma $ for all $n\in\N$.

Let $(\xi_n)_n$ and $(\zeta_n)_n$ be sequences of unit vectors in $\ell_2(Y,H)$ so that, for all $n\in\N$, $\xi_n$ and $\zeta_n$ belong to the range of $\Phi(p_n)$ and $\Phi(q_n)$, respectively.
 Given $ n\in\N$ and $y\not\in A_n$, we have that 
\[\|\Phi(q_n)\chi_{\{y\}}\|^2=\|\chi_{\{y\}}\zeta_n\|^2\leq \|\zeta_n\|^2-\|\chi_{A_n}\zeta_n\|^2\leq 1-\|\Phi(q_n)\chi_{A_n}\|^2< \gamma(2-\gamma).\]
As $\delta=\gamma(2-\gamma)$, this implies that $y\not\in Y_{q_n,\delta}$. As $y$ was an arbitrary element in $X\setminus A_n$, this shows that $Y_{q_n,\delta}\subset A_n$. Analogous arguments applied to $(\xi_n)_n$ and $(B_n)_n$ give us that $Y_{p_n,\delta}\subset B_n$ for all $n\in\N$. Therefore, we must have that \[\partial (Y_{p_n,\delta},Y_{q_n,\delta})\leq s+2s'\]
for all $n\in\N$. This contradicts our choice of $(p_n)_n$ and $(q_n)_n$. 
 \end{proof}

\subsection{Cartan masas}\label{SubsectionProof}
We now prove Theorem~\ref{ThmCartan}; even better, we prove a stronger version of it in Theorem~\ref{ThmIsoMakeAsympCoarseLike}. Recall, given an orthonormal basis $\bar\xi=(\xi_n)_n$ of $H$, we denote by $\ell_\infty(X,\bar \xi)$ the masa of $\cst(X)$ consisting of all operators $a\in\cB(\ell_2(X,H))$ such that for all $x\in X$ there is $(\lambda_n)_n\in c_0$ with 
\[
a(\delta_x\otimes \xi_n)=\lambda_n\delta_x\otimes \xi_n \ \text{ for all }\ n\in\N.
\]
This algebra is a masa in the algebra of those operators in $\cB(\ell_2(X,H))$ such that each entry is locally compact, and therefore it is such in $\cst(X)$.

\begin{definition}
Given metric spaces $X$ and $Y$, a map $\Phi\colon\cst(X) \to \cst(Y)$ is \emph{strongly coarse-like} if for all $r>0$ there is $s>0$ so that $\propg(\Phi(a))\leq s$ for all $a\in \cst(X)$ with $\propg(a)\leq r$.
\end{definition}

 \begin{theorem}\label{ThmIsoMakeAsympCoarseLike}
Let $X$ and $Y$ be u.l.f.\@ metric spaces and assume that $Y$ has property $A$. Let $ \Phi\colon\cst(X)\to\cst(Y)$ be an isomorphism. Let $\bar\xi=(\xi_n)$ and $\bar\zeta=(\zeta_n)$ be orthonormal bases of $H$. Then there exists a unitary $v\in \BD(Y)$ such that $\Ad(v)\circ \Phi:\cst(X)\to\cst(Y)$ is strongly coarse-like and 
 \[
 \Ad(v)\circ \Phi(\ell_\infty(X,\bar \xi))=\ell_\infty(Y,\bar \zeta).
 \]
\end{theorem}

 \begin{proof} 
By \cite[Theorem 4.1]{SpakulaWillett2013AdvMath}, $X$ and $Y$ are coarsely equivalent. Hence $X$ has property A and, as property A and ONL are equivalent for u.l.f.\@ metric spaces \cite[Theorem 4.1]{Sako2014}, we conclude that $X$, $Y$ and $\Phi$ satisfy Assumption \ref{AssumptionIsoONL}. 

For $n\in\N$, let $P_n$ and $Q_n$ be the projections onto $\Span\{\xi_n\}$ and $\Span\{\zeta_n\}$, respectively. For $x\in X$, $y\in Y$ and $n\in\N$, let $p_{x,n}=\chi_{\{x\}}\otimes P_n$ and $q_{y,n}=\chi_{\{y\}}\otimes Q_n$. By \cite[Lemma 4.6]{SpakulaWillett2013AdvMath} (or \cite[Lemma 3.1]{BragaChungLi2019}), there is $\delta>0$ so that $X_{q_{y,1},\delta}$ and $Y_{p_{x,1},\delta}$ are nonempty for all $x\in X$ and all $y\in Y$; such $\delta$ is fixed for the remainder of the proof. Hence, we can pick a map $f\colon X\to Y$ so that $f(x)\in Y_{p_{x,1},\delta}$ for all $x\in X$. By the proof of \cite[Theorem 4.1]{SpakulaWillett2013AdvMath}, $f$ is a coarse equivalence.

Let $X_0\subset X$, $Y_0\subset Y$, $r_0>0$, $(X^x)_{x\in X_0}$ and $(Y^y)_{y\in Y_0}$ be given as in \S\ref{SubsectionCanonicalMap} for $f$, i.e., 
\begin{enumerate}
\item $f\colon X_0\to Y_0$ is a bijection,
\item $X=\bigsqcup_{x\in X_0}X^x$ and $Y=\bigsqcup_{x\in Y_0}Y^y$,
\item $x\in X^x$ and $\diam(X^x)\leq r_0$ for all $x\in X_0$, and
\item $y\in Y^y$ and $\diam(Y^y)\leq r_0$ for all all $y\in Y_0$.
\end{enumerate} 
 
Let $g\colon X\times \N\to Y\times \N$ and $u=u_g\colon\ell_2(X,H)\to \ell_2(Y,H)$ be obtained as in \S\ref{SubsectionCanonicalMap}, i.e., for each $x\in X_0$, $g$ restricts to a bijection $X^x\times \N\to Y^{f(x)}\times \N$ and \[u\delta_x\otimes \xi_n=\delta_{g_1(x,n)}\otimes \zeta_{g_2(x,n)}\]
for all $(x,n)\in X\times \N$. Therefore, the discussion in \S\ref{SubsectionCanonicalMap} implies that $\Psi=\Ad(u)\colon\cst(X)\to \cst(Y)$ is a strongly coarse-like isomorphism.

By \cite[Lemma 3.1]{SpakulaWillett2013AdvMath}, we can pick a unitary $w\colon\ell_2(X,H)\to \ell_2(Y,H)$ so that $\Phi=\Ad(w)$. Let $v=uw^*$, so $\Ad(v)\circ\Phi=\Psi$. Therefore, as noticed above, $\Ad(v)\circ\Phi\colon\cst(X)\to \cst(Y)$ is strongly coarse-like. Moreover, it is clear from the definition of $u$ that 
 \[\Ad(v)\circ \Phi\Big(\ell_\infty\big(X,c_0(\bar \xi)\big)\Big)\subset \ell_\infty\big(Y,c_0(\bar \zeta)\big).\]
Therefore, in order to conclude the proof, we only need to notice that $v\in \BD(Y)$. For that, we now use the technical results proven in \S\ref{SubsectionTechnical}.

\begin{claim}\label{ClaimContainment}
For all $\delta'\in (0,\delta]$, there is $ r>0$ so that for all $y\in Y$ and all rank 1 projections $q\in \cB(\ell_2(\{y\},H))$ we have $Y_{\Psi^{-1}(q),\delta'}\subset B_r(y)$.
\end{claim}

\begin{proof}
Fix $\delta'\in (0,\delta]$. Let $s>0$ be given by Lemma \ref{LemmaDistYpdelta} applied to $\Phi$, $\delta'$ and $r_0$. Lemma \ref{LemmaYxndeltaBoundedDiam} gives $k>0$ so that $\diam(Y_{p,\delta'})\leq k$ for all $p\in \mathrm{Proj}_{1, r_0}(X)$. 

 Fix $y\in Y$ and a rank 1 projection $q\in \cB(\ell_2(\{y\},H))$. Let $y'\in Y_0$ and $x'\in X_0$ be so that $y\in Y^{y'}$ and $y'=f(x')$. Since $g\colon X\times \N\to Y\times \N$ restricts to a bijection $X^{x'}\times \N\to Y^{y'}\times \N$, the definition of $\Psi$ clearly implies that $\Psi^{-1}(q)\in \cB(\ell_2(X^{x'},H))$. Hence, as $p_{x', 1}\in \cB(\ell_2(X^{x'},H))$, $\diam (X^{x'})\leq r_0$, our choice of $s$ implies that $\partial(Y_{\Psi^{-1}(q),\delta'},Y_{p_{x',1},\delta'})\leq s$. By the defining property of $f\colon X\to Y$, we have that $y'= f(x')\in Y_{p_{x',1},\delta}$. As $\delta'\leq \delta$, $y'\in Y_{p_{x',1},\delta'}$. Therefore, as $\diam(Y^{y'})\leq r_0$ we have $\partial (y,y')\leq r_0$, and our choices of $ \ell$ and $k$ imply that 
\[Y_{\Psi^{-1}(q),\delta'}\subset B_{r_0+s+2k}(y)\]
The claim then follows by letting $r=r_0+s+2k$.
\end{proof}

We now show that $v=uw^*$ belongs to $\BD(Y)$. By \cite[Theorem 3.3]{SpakulaZhang2018}, as $Y$ has property A, it is enough to show that $v$ is \emph{quasi-local}.\footnote{An operator $a\in \cB(\ell_2(Y,H))$ is \emph{quasi-local} if for all $\eps>0$ there is $r>0$ so that $\partial(A,B)>r$ implies $\|\chi_Aa\chi_B\|<\eps$ for all $A,B\subset Y$.} Fix $\eps>0$. Let $t>0$ be given by Lemma \ref{Lemma3} for $\eps$, $r_0$ and $\Phi $. Then, for all $k\in\N$ there is $\delta'\in (0,\delta]$ so that 
\begin{align}\label{Eq1}
\|\Phi(\bar p)\chi_{Y\setminus Y_{\bar p,\delta'}}\|\leq \eps 
\end{align}
for all $\bar p=(p_i)_{i\in \N}\in \cD_{t,r_0,k}(X) $.

Before finishing the proof, we introduce some notation: 
 given $C\subset Y$ and a sequence of projections $(Q_y)_{y\in C}$ on $H$ of rank at most 1, we write $q_y=\chi_{\{y\}}\otimes Q_y$ for each $y\in C$.\footnote{Notice that $q_y$ and $Q_y$ (for $y\in C$) are distinct from $q_{y,n}$ and $Q_n$ (for $n\in\N$). We believe that, since the indices are different, this abuse of notation will cause no confusion.} Let 
\[X_C=\Big\{x\in X\mid \exists x'\in X_0\text{ with }x\in X^{x'}\text{ and }C\cap Y^{f(x')}\neq \emptyset\Big\}.\]
By the definition of $\Psi$, it follows that $\Psi^{-1}(q_y)\in \cB(\ell_2(X^{x'},H))$ for all $y'\in Y_0$ with $y\in Y^{y'}$ and $x'=f^{-1}(y')$. Remark \ref{RemarkDktX} then implies that there is $k>0$ so that $\Psi^{-1}(\bar q)=(\Psi^{-1}(q_y))_{y\in C}\in \cD_{t,r_0,k}(X_C)$. Moreover, as noticed in Remark \ref{RemarkDktX}, $k$ depends only on $t$ and $\sup_{y\in Y_0}|Y^{y}|$ (i.e., it depends on neither $C$ nor $(Q_y)_{y\in C}$). Fix such $k$. 

By the defining property of $t$, pick $\delta'\in (0,\delta]$ so that \eqref{Eq1} holds for $k$. Claim \ref{ClaimContainment} gives $r>0$ so that $Y_{\Psi^{-1}(q),\delta'}\subset B_r(y)$ for all $y\in Y$ and all rank 1 projections $q\in \cB(\ell_2(\{y\},H))$. Therefore, if $C\subset Y$ and $(Q_y)_{y\in C}$ is a sequence of projections on $H$ of rank at most 1, it follows that 
\[Y_{\Psi^{-1}(\bar q),\delta}=\bigcup_{y\in C} Y_{\Psi^{-1}(q_y),\delta}\subset B_r(C).\]

 Fix $C\subset Y$, $(Q_y)_{y\in C}$ and $(q_y)_{y\in C}$ as above. Then the previous inequalities give that 
\[\|\Phi(\Psi^{-1}(\bar q))\chi_{Y\setminus B_r(C)}\|\leq \eps.\]
As $\Psi^{-1}=\Ad(u^*)$ and $\Phi=\Ad(w)$, this implies that 
\[\|\bar qv\chi_{Y\setminus B_r(C)}\|=\| \bar quw^*\chi_{Y\setminus B_r(C)}\|=\|wu^*\bar quw^*\chi_{Y\setminus B_r(C)}\|\leq \eps.\]
The arbitrariness of $\bar q=(q_y)_{y\in C}$ (i.e., the arbitrariness of $C\subset Y$ and $(Q_y)_{y\in C}$) gives that $\|\chi_Cv\chi_{Y\setminus B_r(C)}\| \leq \eps$ for all $C\subset Y$.

 Let $A,B\subset Y$ be so that $d(A,B)> r$. Then 
\[
\|\chi_Av\chi_B\|\leq \|\chi_{A\cap Y'}v\chi_{ Y\setminus B_r(A)}\| \leq \eps \]
As $\eps$ was arbitrary, this shows that $v$ is quasi-local, so we are done.
\end{proof}

\begin{remark}
Although we chose to work with fixed bases $\bar \xi$ and $\bar\zeta$ inside each $H$-coordinate of $\ell_2(X,H)$ in Theorem \ref{ThmCartan} (Theorem \ref{ThmIsoMakeAsympCoarseLike}), we chose so merely for simplicity. The same proof holds if for each $x\in X$  we choose basis $\bar \xi^x=(\xi_n^x)_n$ and $\bar\zeta^x=(\zeta_n^x)_n$ of $H$. Precisely, given those choices, let $\prod_{x\in X} c_0(\bar \xi^x)$  be the $\ell_\infty$-sum of $(c_0(\bar \xi^x))_{x\in X}$, where $c_0(\bar \xi^x)$ consists of all $a\in \cB(H)$ so that there is $(\lambda_n)_n\in c_0$ for which  $a \xi^x_n=\lambda_n\xi^x_n$ for all $n\in\N$;  $\prod_{x\in X} c_0(\bar \zeta^x)$ is defined  analogously. Then, if $\Phi\in \Aut(\cst(X))$, there exists a unitary $v\in \BD(X)$ so that 
\[\Ad(v)\circ \Phi\Big(\prod_{x\in X} c_0\big(\bar \xi^x\big)\Big)= \prod_{x\in X} c_0\big(\bar \zeta^x\big).\]
\end{remark}

We are ready to prove Theorem~\ref{thm:main}, which we restate for convenience.
\begin{theorem}\label{ThmIsoCoa}
Let $(X,d)$ be a u.l.f.\@ metric space with property A. The canonical homomorphism 
\[
\Coa(X)\to\Out(\cst(X))
\]
described in \S\ref{SubsectionCanonicalMap} is an isomorphism.
\end{theorem}
\begin{proof}
Let $T\colon\Coa(X)\to\Out(\cst(X))$ be the injective homomorphism constructed in \S\ref{SubsectionCanonicalMap}. Fix $\Phi\in \Aut(\cst(X))$. Let $v\in \BD(X)$ be given by Theorem \ref{ThmCartan} for $\Phi$. Moreover, let $f\colon X\to X$, $u\colon \ell_2(X,H)\to\ell_2(X,H)$ and $w\colon \ell_2(X,H)\to\ell_2(X,H)$ be as in the proof of Theorem \ref{ThmCartan} for $X=Y$ and $\bar\xi=\bar\zeta $. Hence, $\Phi=\Ad(w)$, $v=uw^*$. Notice that $T(f)=[\Ad(u)]$ when the latter is computed in $\Out(\cst(X))$.

We are left to show that $T(f)=[\Phi]$, that is, that $\Ad(u)\circ\Phi^{-1}\in \Inn(\BD(X))$. But this follows since
\[\Ad(u)\circ\Phi^{-1}=\Ad(u)\circ\Ad(w^*)=\Ad(v)\]
and $v\in \BD(X)$. 
\end{proof}

\section{Applications} \label{SecApp}

In this section, we use Theorem \ref{thm:uniform} and Theorem \ref{thm:main} in order to compute --- or at least better understand --- $\Out(\cstu(X))$ and $\Out(\cst(X))$ for some specific spaces $X$. In \S\ref{SubsectionOutX}, \S\ref{SubsectionOutNZ} and \S\ref{SubsectionOutNZn}, we apply our results to $\{n^2\mid n\in\N\}$, $\N^n$, and $\Z^n$, while in \S\ref{SubsectionSolBauSol} and \S\ref{SubsectionLamp}, we work with the solvable Baumslag-Solitar groups $B(1,n)$ and the lamplighter group $F\wr \Z$ for a finite group $F$. For brevity, we skip some definitions in these subsections and refer the reader to an appropriate source. 

\subsection{Outer automorphisms of the (uniform) Roe algebra of $\{n^2\mid n\in\N\}$}\label{SubsectionOutX}
Denote the group of permutations on $\N$ by $S_\infty$. Let $\sim_0$ be the equivalence relation on $S_\infty$ given by $\pi\sim_0\pi'$ if there is $n_0\in\N$ so that $\pi(n)=\pi'(n)$ for all $n\geq n_0$. Clearly, $N=\{\pi\in S_\infty\mid \pi\sim_0 \mathrm{Id}_\N\}$ is a normal subgroup of $S_\infty$, so $S_\infty/{\sim_0}=S_\infty/N$ is a group.

\begin{corollary}
Let $X=\{n^2\mid n\in\N\}$. Then $\mathrm{BijCoa}(X) $ is isomorphic to $S_\infty/{\sim_0}$. In particular, $\Out(\cstu(X)) $ is isomorphic to $S_\infty/{\sim_0}$.
\end{corollary}

\begin{proof}
First, notice that $X$ has property A. For this, notice that $\cstu(X)$ is generated by $\ell_\infty(X)$ and $\mathcal K(\ell_2(X))$. Moreover the uniform Roe corona of $X$, the algebra $\cstu(X)/\mathcal K(\ell_2(X))$ is isomorphic to $\ell_\infty/c_0$, and it is therefore nuclear. Since $\mathcal K(\ell_2(X))$ is nuclear, so is $\cstu(X)$, and therefore $X$ has property $A$. 

A map $f\colon X\to X$ is a bijective coarse equivalence if and only if $f$ is a bijection. So the group of bijective coarse equivalences of $X$ is isomorphic to $S_\infty$. Moreover, two maps $f,g\colon X\to X$ are close if and only if they eventually coincide, i.e., there is $n_0\in \N$ so that $f(n^2)=g(n^2)$ for all $n>n_0$. The result now follows. 
\end{proof}

Denote the group of \emph{cofinite partial bijections} on $\N$ by $S_\infty^*$, i.e., 
\begin{align*}
S_\infty^*=\Big\{(\pi,A,B)\in \N^\N\times \cP(\N)\times \cP(\N)\colon &|A^\complement|, |B^\complement|<\infty \text{ and } \\
& \pi\restriction A\colon A\to B\text{ is a bijection}\Big\}.\end{align*}
By a slight abuse of notation, we denote by $\sim_0$ the equivalence relation on $S^*_\infty$ given by $(\pi,A,B)\sim_0(\pi',A',B')$ if there is $n_0\in\N$ so that $\pi(n)=\pi'(n)$ for all $n\geq n_0$. 

\begin{corollary}
Let $X=\{n^2\mid n\in\N\}$. Then $\mathrm{Coa}(X)/\sim$ is isomorphic to $S^*_\infty/\sim_0$. In particular, $\Out(\cst(X)) $ is isomorphic to $S^*_\infty/\sim_0$.
\end{corollary} 

\begin{proof}
Clearly, a map $f\colon X\to X$ is a coarse equivalence if and only if $f$ there are cofinite $A,B\subset \N$ so that $f\restriction A\colon A\to B$ is a bijection. Moreover, maps $f,g\colon X\to X$ are close if and only if they eventually coincide. The result follows. 
\end{proof}

 \subsection{Outer automorphisms of the uniform Roe algebras of $\N$ and $\Z$}\label{SubsectionOutNZ}

\begin{corollary}\label{CorOutNandZ}
The group $\mathrm{BijCoa}(\N)$ is trivial and $\mathrm{BijCoa}(\Z)$ is isomorphic to $\Z_2$. In particular, $\Out(\cstu(\N)) $ is trivial and $\Out(\cstu(\Z)) $ is isomorphic to $\Z_2$. 
\end{corollary}

\begin{proof}
Fix a bijective coarse equivalence $f\colon\N\to \N$. We are going to prove that $f$ is close to the identity $\mathrm{Id}_\N$. Suppose this is not the case. Then there is a sequence $(x_n)_n$ in $\N$ so that $|f(x_n)-x_n|>n$ for all $n\in\N$. Without loss of generality, we can assume that $(x_n)_n$ is strictly increasing. Moreover, replacing $f$ by $f^{-1}$ if necessary, we can also assume that $f(x_n)+n<x_n$ for all $n\in\N$. For each $n\in\N$, let $z_n=\max\{z\in \N\mid f^{-1} (z)\leq x_n\}$. As $f$ is a bijection and $f(x_n)+n<x_n$, it follows that $z_n-f(x_n)>n$ for all $n\in\N$. As $f^{-1}$ is expanding, it follows that $\lim_{n}(f^{-1}(z_{n}+1)-x_n)=\infty$. So, as $f^{-1}(z_n)\leq x_n$, we have that $\lim_{n}(f^{-1}(z_{n}+1)-f^{-1}(z_n))=\infty$. This contradicts coarseness of $f^{-1}$. This shows that $\mathrm{BijCoa}(\N) $ is the trivial group.

Now fix a bijective coarse equivalence $f\colon\Z\to \Z$. So either $\lim_{n\to \infty}f(n)=\infty$ or $\lim_{n\to \infty}f(n)=-\infty$. Assume that $\lim_{n\to \infty}f(n)=\infty$ and let $x_0=\min(f(\N))$. 

\begin{claim}
There are bijective coarse equivalences $h_1\colon\N\to \N$ and $h_2\colon\Z\setminus \N\to \Z\setminus \N$ so that $f$ is close to the bijection $h_1\cup h_2\colon\Z\to \Z$. 
\end{claim}

\begin{proof}
Notice that $f$ is close to $g=f-x_0$ and that $g(\N)$ is a cofinite subset of $\N$, say $n_0=|\N\setminus g(\N)|$. Pick bijections \[i\colon\{-n_0,\ldots, -1\}\to \N\setminus g(\N)\] and \[j\colon g^{-1}(\N\setminus g(\N))\to g(\{-n_0,\ldots,-1\}),\] and notice that $g$ is close to 
\[h(x)=\left\{\begin{array}{ll}
g(x),& x\in \N,\\
i(x), & x\in \{-n_0,\ldots, -1\},\\
j(x),& x\in g^{-1}(\N\setminus g(\N)),\\
\end{array}\right.\]
For each $x\in \Z$, let $h_0(x)=h(x-n_0)$ and let $h_1=\restriction \N$ and $h_2=h\restriction \Z\setminus \N$. Since $h$ is close to $h_0$, the result follows. 
 \end{proof}

 Let $h_1$ and $h_2$ be given by the claim above. As $\mathrm{BijCoa}(\N) $ is the trivial group, it follows that $h_1$ is close to $\mathrm{Id}_\N$ and $h_2$ is close to $\mathrm{Id}_{\Z\setminus \N}$ . So $f$ is close to $\mathrm{Id}_\Z$. 
 
If $\lim_{n\to \infty}f(n)=-\infty$, then proceeding analogously as above, we obtain bijective coarse equivalences $h_1\colon\N\to \Z\setminus \N$ and $h_2\colon\Z\times \N\to \N$ so that 
\begin{enumerate}
\item $h_1$ is close to the map $x\in \N\to -x-1\in \Z\setminus \N$,
\item $h_2$ is close to the map $x\in \Z\setminus \N\to -x-1\in \N$, and 
\item $f$ is close to $h_1\cup h_2$.
\end{enumerate}
As $ h_1\cup h_2$ is close to $-\mathrm{Id}_\Z$, so is $f$. 

We have then shown that $\mathrm{BijCoa}(\Z) $ is isomorphic to $\{-\mathrm{Id}_\Z,\mathrm{Id}_\Z\}$. This completes the proof.

The last statement follows from the above and Theorem \ref{thm:uniform}.
\end{proof}

 \subsection{Outer automorphisms of the Roe algebra of $\Z^n$}\label{SubsectionOutNZn}
Recall, given metric spaces $(X,d)$ and $(Y,\partial)$, a map $f\colon X\to Y$ is a \emph{coarse Lipschitz equivalence}\footnote{Coarse Lipschitz equivalences are also referred to as \emph{quasi-isometries} in the literature.} if it is cobounded and there is $L>0$ so that 
\[L^{-1}d(x,y)-L\leq \partial(f(x),f(y))\leq Ld(x,y)+L\]
for all $x,y\in X$. Define 
\[\mathrm{CoaLip}(X)=\Big\{f\colon X\to X\mid f\text{ is a coarse Lipschitz equivalence}\Big\}/{ \sim },\]
where $\sim$ is the closeness relation on functions $X\to X$. Clearly, $\mathrm{CoaLip}(X)$ is a group under composition, i.e., $[f]\circ[g]=[f\circ g]$.

Given $n\in\N$, the inclusion $\Z^n\hookrightarrow \R^n$ is a coarse equivalence (coarse Lipschitz equivalence even). Therefore $\Coa(\Z^n)\cong \Coa(\R^n)$. Moreover, notice that a map $ \R\to \R$ is coarse if and only if it is coarse Lipschitz \cite[Theorem 1.4.13]{NowakYuBook}. Therefore, we have that \[ \Coa(\Z^n)\cong \mathrm{CoaLip}(\R^n).\]
As a consequence of that, results in the literature give us the next corollaries of Theorem \ref{thm:main}. We denote by $\mathrm{PL}_\delta(\Z)$ the group of piecewise linear homeomorphisms $f\colon\R\to \R$ so that $\{|f'(x)|\mid x\in \R\}\subset [M^{-1},M]$ for some $M>0$. Modding out by the closeness relation, we obtain the group $\mathrm{PL}_\delta(\Z)/\sim $.
 
\begin{corollary}
The group $\Out(\cst(\Z))$ has trivial center. Moreover, $\Out(\cst(\Z))$ is isomorphic to $\mathrm{PL}_\delta(\Z)/{\sim}$.
\end{corollary}

\begin{proof}
This follows immediately from Theorem \ref{thm:main}, the discussion preceding the corollary, 
\cite[Theorem 1.1]{Chakraborty2019IndPureApplMath}, and \cite[Theorem 1.2]{Sankaran2006PAMS}.
\end{proof}

\begin{corollary}\label{CorThomp}
The group $\Out(\cst(\Z))$ contains isomorphic copies of the following groups:
\begin{enumerate}
\item $\mathrm{PL}_\kappa(\R)$, the group of piecewise linear homeomorphisms $f\colon\R\to \R$ so that $\overline{\{x\in \R\mid f(x)\neq x\}}$ is compact, 
\item Thompson's group $F$,\footnote{We refer the reader to \cite[Section 1]{CannonFloydParry1996EnseignMath} for the definition of Thompson's group $F$.} and 
\item the free group of rank the continuum.
\end{enumerate}
\end{corollary}

\begin{proof}
This follows immediately from Theorem \ref{thm:main}, the discussion above and \cite[Theorem 1.3]{Sankaran2006PAMS}.
\end{proof}
 
Given a metric space $(X,d)$, a map $f\colon X\to X$ is a bi-Lipschitz equivalence if there is $L>0$ so that 
\[L^{-1}d(x,y)\leq d(f(x),f(y))\leq Ld(x,y)\]
for all $x,y\in X$. We let \[\mathrm{BiLip}(X)=\Big\{f\colon X\to X\mid f\text{ is a bi-Lipschitz equivalence}\Big\}.\]
 So $\mathrm{BiLip}(X)$ is a group under composition (notice that we do not mod out the bi-Lipschitz equivalences by closeness). 
 
 \begin{corollary}
Given $n\in\N$, the group $\Out(\cst(\Z^n))$ contains isomorphic copies of the following groups:
\begin{enumerate}
\item $\mathrm{BiLip}(\mathbb S^{n-1})$, where $\mathbb S=\{z\in \C\mid |z|=1\}$, and 
\item $\mathrm{BiLip}(\mathbb D^n,\mathbb S^{n-1})$.
\end{enumerate}
\end{corollary}
 
 \begin{proof}
 This follows immediately from Theorem \ref{thm:main}, the discussion above and \cite[Theorem 1.1]{Mitraankaran2019TopAppl}.
 \end{proof}

 \subsection{Solvable Baumslag-Solitar groups}\label{SubsectionSolBauSol}
 Given $n\in\N$, $\Q_n$ denotes the $n$-adic rationals and $B(1,n)$ denotes the \emph{solvable Baumslag-Solitar group}, i.e., the group generated by elements $a$ and $b$ subject to the relation $aba^{-1}=b^n$. We endow $B(1,n)$ with the metric given by its Cayley graph structure (see \cite[Definition 1.2.7]{NowakYuBook} for definitions).

 \begin{corollary}\label{CorBaumslagSolitar}
Given $n\in\N$, the group $\Out(\cst(B(1,n)))$ is isomorphic to $\mathrm{BiLip}(\R)\times \mathrm{BiLip}(\Q_n)$.
 \end{corollary}
 
 \begin{proof}
 Since $B(1,n)$ is solvable, it is amenable. Therefore, as $B(1,n)$ is finitely generated, it must also have property A \cite[Theorem 4.14]{NowakYuBook}. By Theorem \ref{thm:main}, we only need to compute $\Coa(B(1,n))$. As $B(1,n)$ is a finitely generated group, we have that $\Coa(B(1,n)) =\mathrm{CoaLip}(B(1,n))$ \cite[Corollary 1.4.15]{NowakYuBook}. The result then follows since it is known that $\mathrm{CoaLip}(B(1,n))$ is isomorphic to $\mathrm{BiLip}(\R)\times \mathrm{BiLip}(\Q_n)$ (see \cite[Theorem 8.1]{FarbMosher1998Inventiones}).
 \end{proof}

 \subsection{The lamplighter group $F\wr \Z$}\label{SubsectionLamp}
Given a group $F$, we denote the wreath product of $F$ and $\Z$ by $F\wr \Z$ (we refer the reader to \cite[Definition 2.6.2]{NowakYuBook} for a precise definition). This group is commonly called the \emph{lamplighter group $F\wr \Z$}, and we endow $F\wr \Z$ with the metric given by its Cayley graph structure \cite[Definition 1.2.7]{NowakYuBook}.

Consider the semidirect product $(\mathrm{BiLip}(\Q_n)\times \mathrm{BiLip}(\Q_n))\rtimes \Z_2$ given by the action of $\Z_2$ on $\mathrm{BiLip}(\Q_n)\times \mathrm{BiLip}(\Q_n)$ of switching factors.\footnote{Recall that if $N$ and $H$ are groups and $\alpha\colon H\curvearrowright N$ is an action, then $N\rtimes H$ denotes the \emph{semidirect product}, i.e., the set $N\times H$ endowed with the product $(n,h)\cdot (n',h')=(n\alpha(h)n' ,hh')$.}

\begin{corollary}\label{CorLamplighter}
Let $F$ be a group with $|F|=n$. Then the group $\Out(\cst(F\wr \Z))$ is isomorphic to \[\Big(\mathrm{BiLip}(\Q_n)\times \mathrm{BiLip}(\Q_n)\Big)\rtimes \Z_2\]
\end{corollary}

\begin{proof}
Since $F$ and $\Z$ are finitely generated, so is $F\wr \Z$ \cite[Chapter 2, Exercise 2.5]{NowakYuBook}. The lamplighter group $F\wr \Z$ is amenable (see \cite[Corollary 2.5]{Woess2013IntMathNach}), and since it is finitely generated, it has property A \cite[Theorem 4.14]{NowakYuBook}. Therefore, we only need to compute $\Coa(F\wr \Z) $ (Theorem \ref{thm:main}). As $F\wr \Z$ is finitely generated, $\Coa(F\wr \Z) =\mathrm{CoaLip}(F\wr \Z)$, by \cite[Theorem 1.4.13]{NowakYuBook}. Moreover, $F\wr \Z$ is quasi-isometric to the Diestel-Leader graph $D(n,n)$ (we refer to \cite[Section 1]{EskinFisherWhyte2012Annals} for both the definition of $\mathrm{DL}(n,n)$ and this fact). So it is enough to compute $\mathrm{CoaLip}(\mathrm{DL}(n,n))$. The result then follows since it is known that $\mathrm{CoaLip}(\mathrm{DL}(n,n))$ is isomorphic to $(\mathrm{BiLip}(\Q_n)\times \mathrm{BiLip}(\Q_n))\rtimes \Z_2$ \cite[Theorem 2.1]{EskinFisherWhyte2012Annals} (see the discussion at the end of \cite[Section 2]{EskinFisherWhyte2012Annals}).
\end{proof}

\begin{acknowledgements*}
The current paper started from a question asked by Ralf Meyer to the authors about whether Theorem \ref{thm:main} was true for the metric space $\Z^n$. The authors would like to thank Ralf Meyer for the interesting question and for several comments on a previous version of this paper. AV is partially supported by the ANR Project AGRUME (ANR-17-CE40-0026).
\end{acknowledgements*}

\end{document}